\documentclass[leqno, letterpaper, 11pt]{article}
\usepackage{amsmath}
\usepackage{amsthm}
\usepackage{amsfonts} % For AMS Beautification
\usepackage{fullpage} % For More Realistic Page Usage

\usepackage{amsmath}
\usepackage{amsfonts} % For AMS Beautification
\usepackage{graphicx}
\usepackage{amssymb}
\usepackage{float}
\usepackage{epstopdf}
\usepackage{tikz-cd}
\usepackage{tikz}
\usepackage{algpseudocode}
\usepackage{multicol}
\usepackage{thm-restate}
\usetikzlibrary{arrows, automata}
\usepackage{CJKutf8}
\usepackage{subfig}
\usepackage{amsthm, thmtools}
\usepackage{bm}
\usepackage{mathrsfs}
\usepackage[lined, commentsnumbered]{algorithm2e}

\newcommand{\N}{\mathbb{N}}

\newcommand{\Z}{\mathbb{Z}}

\newcommand{\cO}{\mathcal{O}}

\newcommand{\ord}{\text{\rm ord}}
\newcommand{\lcm}{\text{\rm lcm}}

\newcommand{\Poi}{\text{\rm Poisson}}
\newcommand{\Bin}{\text{\rm Binomial}}

\newcommand{\Tau}{\mathcal{T}}

\newcommand{\rank}{\text{\rm rank}}
\newcommand{\Cy}{\zeta}

\newtheorem{thm}{Theorem}[section]
\newtheorem{prop}[thm]{Proposition}
\newtheorem{defi}[thm]{Definition} 
\newtheorem{lem}[thm]{Lemma} 
 
\newtheorem{alg}[thm]{Algorithm} 
\newtheorem{ques}[thm]{Question} 
\newtheorem{con}[thm]{Conjecture}

\newcommand\blfootnote[1]{%
  \begingroup
  \renewcommand\thefootnote{}\footnote{#1}%
  \addtocounter{footnote}{-1}%
  \endgroup
}

\begin{document}

\title{Weakly robust periodic solutions of one-dimensional cellular automata with random rules}
 
%\title{One-dimensional cellular automata:\\
%periodic solutions with long temporal periods}
\author{{\sc Janko Gravner} and {\sc Xiaochen Liu}\\
Department of Mathematics\\University of California\\Davis, CA 95616\\{\tt gravner{@}math.ucdavis.edu, xchliu{@}math.ucdavis.edu}}
\maketitle
\bibliographystyle{plain}

\begin{abstract}
We study $2$-neighbor one-dimensional cellular automata with a large number $n$ of states and randomly selected rules.
We focus on the rules with weakly robust periodic solutions (WRPS). 
WRPS are global configurations that exhibit spatial and 
temporal periodicity and advance into any environment with at least a fixed strictly positive velocity. Our main result quantifies how unlikely WRPS are:
the probability of existence of a WRPS within a finite range of periods is asymptotically proportional to $1/n$, provided that 
a divisibility condition is satisfied.
Our main tools come from random graph theory and the Chen-Stein method for Poisson approximation.
\end{abstract}


\section{Introduction} \label{section: introduction}
We continue our study of one-dimensional cellular automata (CA) with random rules, initiated in~\cite{gl1}.
As in that paper, we investigate rules with $n$ states and $2$ neighbors, with a rule chosen uniformly from all $n^{n^2}$ rules. 
In~\cite{gl1}, we provided the asymptotic probability, as 
$n$ goes to infinity, that such a rule has a periodic solution (PS) 
with a 
given spatial and temporal period.
In this paper, we 
demand a certain additional stability property of a PS, and explore 
the analogous probability of existence of a PS in this special class. 

To be precise, we consider one-dimensional cellular automata with 
$n$-state space, encoded by $\Z_n = \{0, \dots, n - 1\}$, and $2$-neighbor rules $f: \Z_n^2 \to \Z_n$.
Assume that a CA given by the rule $f$ starts from a periodic global configuration $\xi_0: \Z \to \Z_n$ that satisfies $\xi_0(x) = \xi_0(x + \sigma)$, for all $x \in \Z$.
If we also have $\xi_\tau = \xi_0$, and $\tau$ and $\sigma$ are both minimal, then we have found a \textbf{periodic solution} (PS) under rule $f$, with \textbf{spatial period} $\sigma$ and 
\textbf{temporal period} $\tau$. We will not distinguish between 
spatial and temporal shifts of a PS. Therefore, 
each configuration $\xi_t\in \Z_n^{\Z}$, $t\ge 0$, characterizes the PS and is called a \textbf{PS configuration}.
We call the map $(x, t) \mapsto \xi_t(x)$ from $\Z \times \Z_+$ to $\Z$ the \textbf{space-time configuration}; within it, 
any rectangle with $\tau$ rows and $\sigma$ columns also  characterizes the PS, and we call any such rectangle the \textbf{tile} of the PS.
Thus we do not distinguish between tiles which are spatial or
temporal rotations of each other.

In the present paper, we are interested in PS with an expansion property, which we first illustrate by an example and provide 
some motivation, and then give a 
formal definition.
Figure~\ref{figure: wrps example} demonstrates two pieces of the space-time configurations under the 3-state rule\textit{102222210}. (As in~\cite{gl1}, we name a rule by listing its values for all pairs in reverse alphabetical order from $(n - 1, n - 1)$ to $(0, 0)$.)
The tile
$$
\begin{matrix}
0 & 2 & 2 & 2 & 1 & 1\\
2 & 2 & 1 & 1 & 0 & 2\\
1 & 1 & 0 & 2 & 2 & 2\\
\end{matrix}
$$
characterizes a PS under this rule and for such PS, even if the 
spatially periodic configuration is replaced by an arbitrary configuration to the right of some site in $\Z$, the periodic configuration will ``repair'' itself, that is, it will advance to the right with a minimal velocity $v > 0$ as time increases, uniformly over the perturbed environment.

PS with such property are of particular interest and importance, as they are related to stable limit cycles in continuous dynamical systems.
Limit cycles, also known as isolated closed trajectories, are such that neighboring trajectories either spiral toward or away from them.
In the former case, when a perturbation of a limit cycle converges back,
the limit cycle is called stable~\cite{strogatz2001nonlinear}.
Thus we consider an analogous stability property for CA: 
after a one-sided perturbation of a periodic configuration, 
the dynamics make the configuration converge back.
In this paper, we keep the terminology from~\cite{gravner2012robust} 
and refer to 
such stability as robustness. We remark that the minimal velocity $v$ gives the minimal exponential 
rate of convergence to the PS in the standard metric, by which
the distance between $\xi, \eta\in\Z_n^\Z$ is $\mathfrak{m}(\xi, \eta) = 2^{-n}$, where $n = \inf \{|x| : \xi(x) \neq \eta (x)\}$. 

\begin{figure}[ht] 
    \centering
	\includegraphics[scale = 0.1, trim = 5cm 1cm 3cm 1cm, clip]{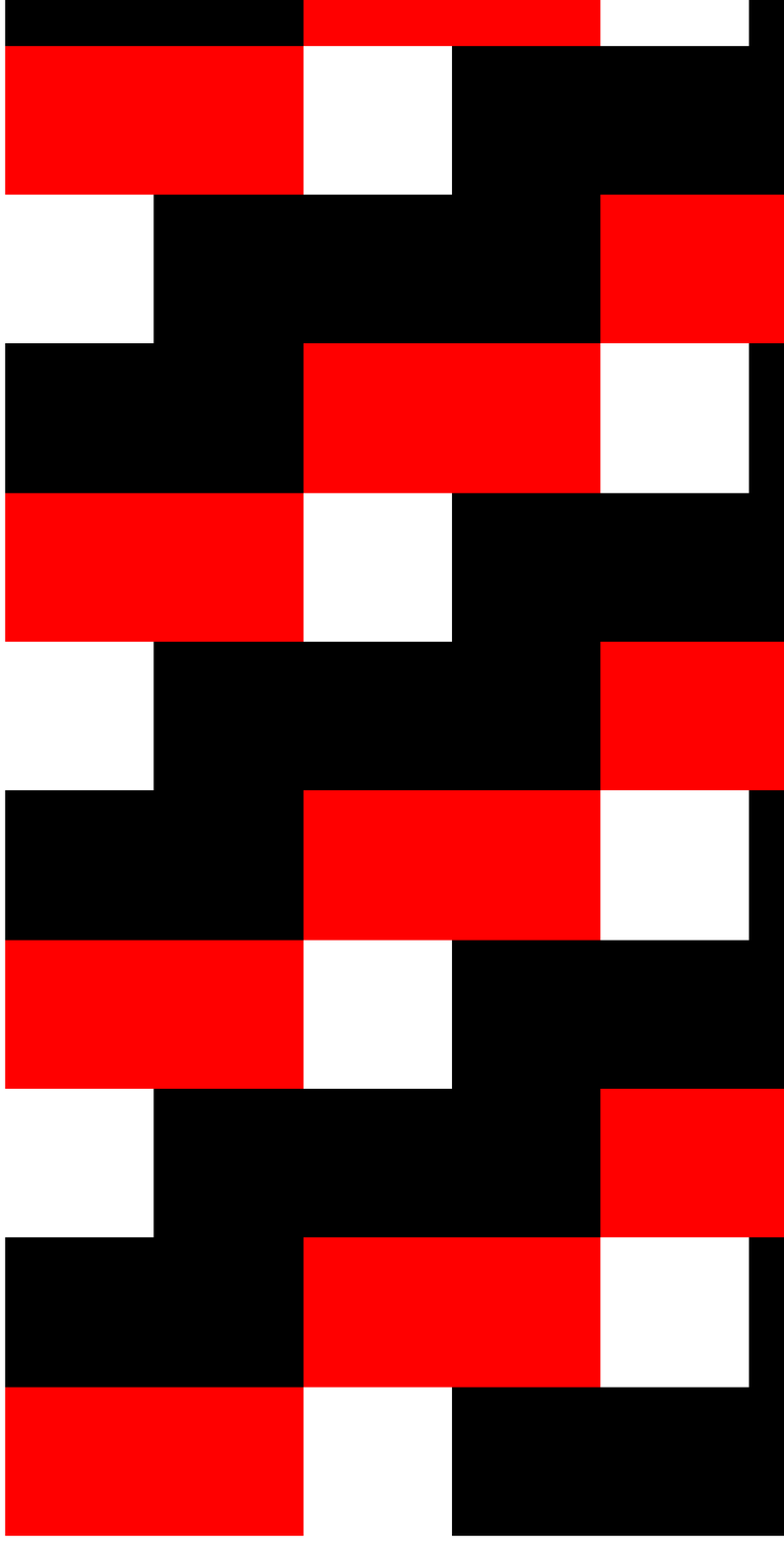}
	\hspace{1cm}
	\includegraphics[scale = 0.1, trim = 5cm 1cm 3cm 1cm, clip]{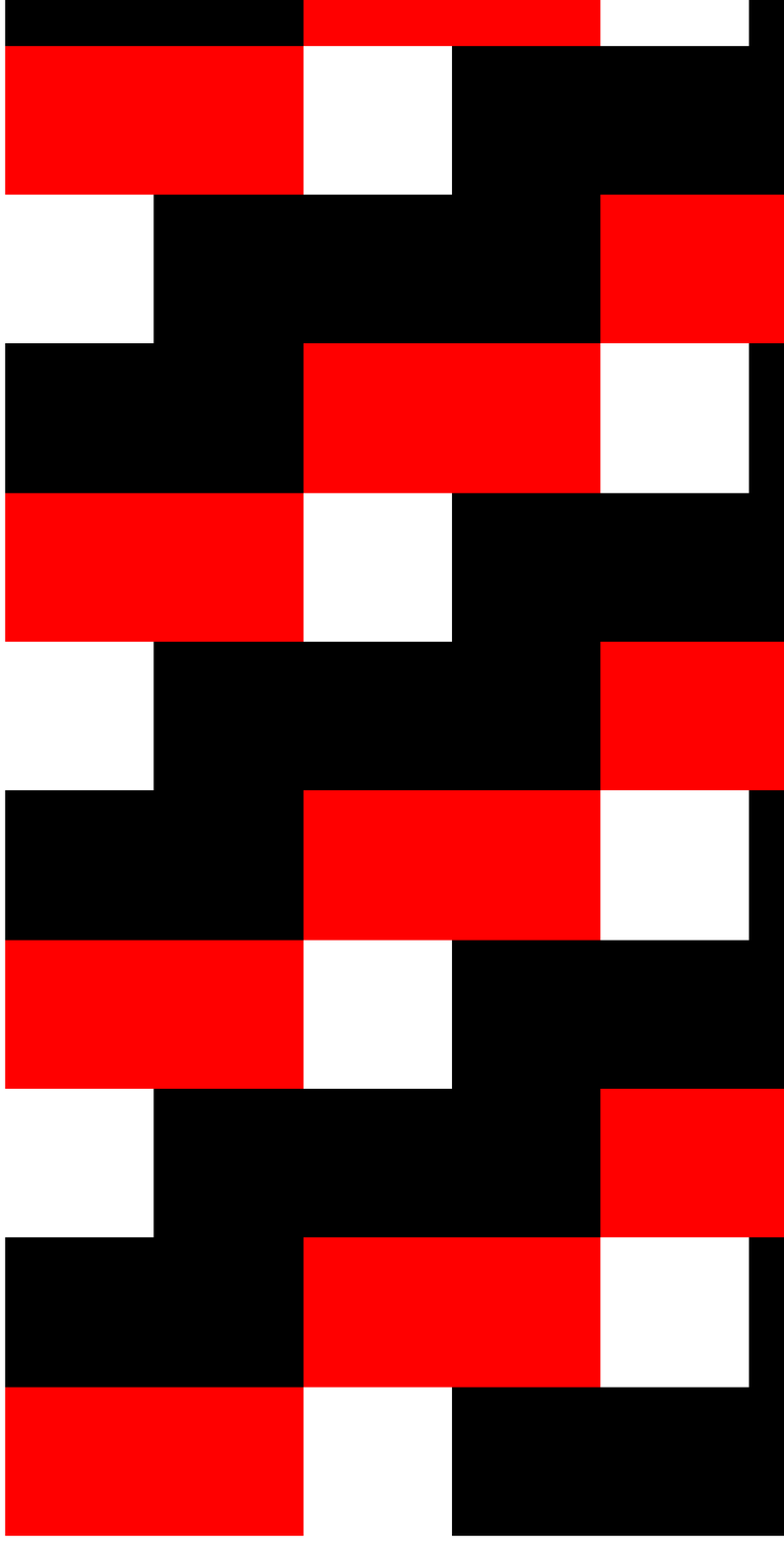}    
    \\
    \caption{Two pieces of the space-time configuration of the 3-state rule\textit{102222210}. 
    The underlying PS exhibits weak robustness: the periodicity expands if it is terminated and continued by an arbitrary configuration, 
    for example a random configuration (left) or all $0$s (right).}
           \label{figure: wrps example}
\end{figure}

Proceeding to the formal definition, let $\xi_0$ be a PS configuration under rule $f$ and $\eta_0$ be any initial configuration that agrees with $\xi_0$ on all $x \le y$, for some $y \in \Z$.
Adapting the definition from~\cite{gravner2012robust}, we call
such initial configurations \textbf{proper} for the PS $\xi_0$.
Let $\xi_t$ and $\eta_t$ be the configurations obtained by running $f$ starting with $\xi_0$ and $\eta_0$, respectively.
%Let
%$$s_t(\eta_0) = \sup \{y : \eta_t(x) = \xi_t(x), \text{ for all } x < y\}$$
%be the first location that $\eta_t$ does not agree with $\xi_t$ at time $t$.
%
Let
$$s_t(\eta_0) = \sup \{x\in \Z: \eta_t(x) = \xi_t(x)\}$$
be the rightmost location at which $\eta_t$ agrees with $\xi_t$ at time $t$.  
Then the \textbf{expansion velocity in the initial environment} $\eta_0$ is 
$$v(\eta_0) = \liminf_{t \to \infty}\frac{s_t}{t},$$
which describes the rate at which spatial periodicity expands. 
The \textbf{expansion velocity}
$$v = \inf \{v(\eta_0): \eta_0 \text{ is proper for } \xi_0\}$$
then measures uniformity over all environments. If $v>0$, then 
the PS $\xi_t$ is \textbf{weakly robust}. With this terminology, 
we distinguish this property from the more restrictive 
robustness from~\cite{gravner2012robust}.

As in~\cite{gl1}, we are interested in the existence of WRPS of a randomly selected $n$-state $2$-neighbor rule $f$. To this end,
fix two sets $\Tau, \Sigma\subset \N=\{1,2,\ldots\}$ and 
let $\mathcal{R}_{\Tau, \Sigma}$ be the (random) set of WRPS of a randomly selected $n$-state rule $f$, with temporal period $\tau$ and spatial period $\sigma$ satisfying 
$(\tau, \sigma) \in \Tau \times \Sigma$. 
While our results for existence of PS~\cite{gl1} are valid 
for arbitrary finite
$\Tau \times \Sigma \subset \mathbb{N} \times \mathbb{N}$, we impose 
a divisibility restriction for our result on WRPS.

\begin{thm} \label{theorem: main}
Let $\Tau \times \Sigma \subset \mathbb{N} \times \mathbb{N}$ be fixed and finite.
If there exists $(\tau, \sigma) \in \Tau \times \Sigma$ such that $\sigma \mid \tau$, then $\mathbb{P}(\mathcal{R}_{\Tau, \Sigma} \neq \emptyset) = c(\Tau, \Sigma)/n + o(1/n)$, where $c(\Tau, \Sigma)$ is a constant depending only on $\Tau$ and $\Sigma$.
\end{thm}

In addition to~\cite{gl1}, we have investigated periodic solutions
for cellular automata in~\cite{gl2, gl3}, where the emphasis is on maximal temporal periods; some further results 
and conjectures on robustness are in~\cite{xiaochen}.
The initial motivation for the present paper comes 
from the investigation of robust periodic solutions (RPS) in \cite{gravner2012robust}, in which all the 64 one-dimensional binary 3-neighbor edge CA rules and their RPS are studied. To our knowledge, 
robustness of PS is first addressed for the \textit{Exactly 1} rule, i.e., the elementary CA \textit{Rule 22}, in~\cite{gravner2011one}.

This paper is organized as follows.
In the next section, we recall some preliminary results from~\cite{gl1}.
While we summarize major definitions and tools, we omit the proofs and refer the reader to~\cite{gl1} for a more detailed discussion.
In Section~\ref{section: decidability and wrps}, we introduce the property of a tile that distinguishes a WRPS from a PS, i.e., the decidability of labels in a tile.
We establish the probability that a label exhibits such property for a randomly selected rule in Section~\ref{section: decidability probability} and give the proof of Theorem~\ref{theorem: main} in Section~\ref{section: proof of theorem}.
In the final section, we discuss the possible directions and methods to extend and generalize our results.
 
%\note{Up to here.}\note{Up to here.}

\section{Preliminaries} \label{section: preliminaries}
The main purpose of this section is to gather 
the relevant definitions and results from~\cite{gl1}. All lemmas
are restatements of results in~\cite{gl1}, where the proofs
are provided.  

\subsection{Tiles of PS} \label{subsection: tiles of a ps}
%First, we recall more properties of tiles.
We may express a tile with periods $\tau$ and $\sigma$ as $T = (a_{i,j})_{i = 0,\dots, \tau-1, j = 0, \dots ,\sigma - 1}$, once we fix an element in $T$ to be placed at the position $(0, 0)$.
We use the notation $\mathtt{row}_i$ and $\mathtt{col}_j$ to denote the $i$th row and $j$th column of a tile $T$ and use $a_{i, j}$ to denote the element at the $i$th row and $j$th column of $T$, where we always interpret the two subscripts modulo $\tau$ and $\sigma$, respectively.
%All the properties we now introduce are independent of the chosen rotation (as they must be, to be meaningful).

Let $T_1$ and $T_2$ be two tiles and $a_{i, j}$, $b_{k, m}$ be the corresponding elements. 
If $\left(a_{i, j}, a_{i, j + 1}\right) \neq \left(b_{k, m}, b_{k, m + 1}\right)$ for $i, j, k, m \in \Z_+$, then $T_1$ and $T_2$ are called \textbf{orthogonal}, denoted by $T_1 \perp T_2$.
In this case, we observe that two assignments $(a_{i, j}, a_{i, j + 1}) \mapsto a_{i + 1, j + 1}$ and $(b_{k, m}, b_{k, m + 1}) \mapsto b_{k + 1, m + 1}$ occur independently.
We say that $T_1$ and $T_2$ are \textbf{disjoint}, and denote this property by $T_1 \cap T_2 = \emptyset$, if $a_{i, j} \neq b_{k, m}$, for $i, j, k, m \in \Z_+$.
Clearly, every pair of disjoint tiles is orthogonal, but not vice versa.

The following quantities associated with a tile play a important role in the sequel.
We define the \textbf{assignment number} of $T$ to be $p(T) = \#\{(a_{i, j}, a_{i,j + 1}): a_{i, j}, a_{i, j + 1} \in T\}$, i.e., the number of values of the rule $f$ specified by $T$.
Also, let $s(T) =  \#\{a_{i, j}: a_{i, j} \in T\}$ be the number of different states in the tile. 
Clearly, $p(T) \ge s(T)$, so we define $\ell = \ell(T) = p(T) - s(T)$ to be the $\textbf{lag}$ of $T$. 

The following lemma from~\cite{gl1} lists two immediate properties of the tile of a PS.

\begin{lem} \label{lemma: properties of tile}
Let $T= (a_{i, j})_{i = 0,\dots, \tau - 1, j = 0, \dots ,\sigma - 1}$ be the tile of a PS with periods $\tau$ and $\sigma$.
Then $T$ satisfies the following properties: 
\begin{enumerate}
\item Uniqueness of assignment: if $(a_{i, j}, a_{i, j + 1}) = (a_{k, m}, a_{k, m + 1})$, then $a_{i + 1, j + 1} = a_{k + 1, m + 1}$. 
\item Aperiodicity of rows: each row of $T$ cannot be divided into smaller identical pieces.
\end{enumerate}
\end{lem}

%\begin{proof}
%See~\cite{gl1}.
%%Part 1 is clear since $T$ is generated by a CA rule. 
%%Part 2 follows from part 1 and the assumption that the spatial period of $T$ is minimal.
%\end{proof}

We remark that for a tile of a PS that is not weakly robust, there {\it may\/} exist periodic columns.
%For example, note that the first column in Figure~\ref{figure: 4 examples}(d) has period $2$ rather than $4 = \tau$. 
However, in Section~\ref{section: decidability and wrps}, we will show that, if $T$ is a tile of a WRPS, its columns are necessarily aperiodic.

\subsection{Circular Shifts} \label{subsection: circular shifts}
We also recall the concept of circular shifts operation on $Z_n^\sigma$ (or $Z_n^\tau$), the set of words of length $\sigma$ (or $\tau$) from the alphabet $\Z_n$, which will be used in Section~\ref{subsection: simple tiles}.

\begin{defi} \label{defination: circular shift}
Let $\Z_n^\sigma$ consist of all length-$\sigma$ words.
A \textbf{circular shift} is a map $\pi: \Z_n^\sigma \to \Z_n^\sigma$, given by an $i \in \Z_+$ as follows: $\pi(a_0 a_1\dots a_{\tau - 1}) =  a_i a_{i + 1} \dots a_{i + \sigma - 1}$, where the subscripts are modulo $\sigma$.
The \textbf{order} of a circular shift $\pi$ is the smallest $k$ such that $\pi^k(A) = A$ for all $A \in \Z_n^\sigma$, and is denoted by $\ord(\pi)$.
Circular shifts on $\Z_n^{\tau}$ will also appear in the sequel and are defined in the same way.
\end{defi}

\begin{lem} \label{lemma: euler totient}
Let $\pi$ be a circular shift on $\Z_n^\sigma$ and let $A\in \Z_n^\sigma$ be an aperiodic length-$\sigma$ word from alphabet $\Z_n$. 
Then:
(1) $\ord(\pi) \mid \sigma$; and
(2) for any $d \mid \sigma$,
$$\#\left\{B \in \Z_n^\sigma: A = \pi(B) \text{ for some }\pi \text{ with }\ord(\pi) = d \right\} = \varphi(d).$$ 
\end{lem}
%\begin{proof}
%See~\cite{gl1}.
%%Note that the $\sigma$ circular shifts form a cyclic group of order $\sigma$.
%%Moreover, $\ord(\pi)$ of a circular shift is its order in the group, thus (1) follows.
%%To prove (2), observe that the circular shifts of order $d$ generate a cyclic subgroup and the number of them is $\varphi(d)$.
%%As $A$ is aperiodic, the cardinality in the claim is the same.
%\end{proof}

Two words $A$ and $B$ of length $\sigma$ are \textbf{equal up to a circular shift} if $B = \pi(A)$ for some circular shift $\pi$.
%For example, words $0123$ and $2301$ are not equal, but are equal up to a circular shift. 

\subsection{Directed Graph on Labels} \label{subsection: directed graph on labels}

In our study of PS~\cite{gl1}, we extended the notion of label trees from~\cite{gravner2012robust} to define the \textbf{label digraph}.
As this object is also of relevance to WRPS, we recall its definition in this subsection.
 
\begin{defi} \label{definition: right-extend}
Let $A = a_0 \dots a_{\tau - 1}$ and $B=b_0 \dots b_{\tau - 1}$ be two words from alphabet $\Z_n$, which we call \textbf{labels} of length $\tau$.
(While it is best to view them as vertical columns, we write them horizontally for reasons of space, as in~\cite{gravner2012robust}.)
We say that $A$ \textbf{right-extends to} $B$ if $f(a_{i}, b_{i}) = b_{i + 1}$, for all $i \in \Z_+$, where (as usual) the indices are modulo $\tau$, and we write $A \to B$.
We form the \textbf{label digraph} associated with a given $\tau$ by forming an arc from a label $A$ to a label $B$ if $A$ right-extends to $B$.
\end{defi}

The right extension relation is the basis for the Algorithm~\ref{algorithm: ps from label digraph} below for finding all the PS with temporal period $\tau$.

\iffalse
A label $A = a_0 \dots a_{\tau - 1}$ right-extends to $B$ if and only if we preserve the temporal periodicity from a column $A$ to the column $B$ to its right.
This fact is the basis for the Algorithm~\ref{algorithm: ps from label digraph} below, which gives all the PS with temporal period $\tau$.
The label digraph of same rule as in Figure~\ref{figure: configuration digraph} and temporal period $\tau=2$ is presented in Figure~\ref{figure: label digraph}.
For example, we have the arc from label 12 to 10 as $1 \underline{1} \mapsto 0$, $2\underline{0} \mapsto 1$. 
Either of the two 3-cycles in the digraph generates the PS in Figure~\ref{figure: ps example}. 
\fi

\begin{alg}\label{algorithm: ps from label digraph}
\quad \\
\noindent\rule{\textwidth}{1pt}
\begin{algorithm}[H] %\label{algorithm: generating random DEC}
\SetKwInOut{Input}{input}
\SetKwInOut{Output}{output}
\SetKw{Print}{print}
\SetKw{Return}{return}
\DontPrintSemicolon

\Input{
Label digraph $D_{\tau, f}$ of $f$ with temporal period $\tau$\;
}
\;
Find all the directed cycles in $D_{\tau, f}$\;
\For{each cycle $A_0 \to A_1 \to \cdots \to A_{\sigma - 1}\to A_0$}{form the tile $T$   
by placing labels $A_0, A_1,\ldots,  A_{\sigma - 1}$ on successive columns.\;
	\If{both spatial and temporal periods of $T$  are minimal}{
		\Print $T$ as a PS\;
	}
}
\end{algorithm}
\noindent\rule{\textwidth}{1pt}
\end{alg}

\iffalse
\begin{alg} \label{algorithm: ps from label digraph}
Input: Label digraph $D_{\tau, f}$ of $f$ with period $\tau$.\\
Step 1: Find all the directed cycles in $D_{\tau, f}$.\\
Step 2: For each cycle $A_0 \to A_1 \to \cdots \to A_{\sigma - 1}\to A_0$, form the tile $T$   
by placing configurations $A_0, A_1, \ldots,  A_{\sigma - 1}$ on successive columns.\\
Step 3: If both spatial and temporal periods of $T$  are minimal, then output $T$.
\end{alg}
\fi
\begin{prop}
All PS of temporal period $\tau$ of $f$ can be obtained by the Algorithm~\ref{algorithm: ps from label digraph}.
\end{prop}
\subsection{Chen-Stein Method for Poisson Approximation} \label{subsection: chen-stein method for poisson approximation}
The most useful tool in proving Poisson convergence is the Chen-Stein method~\cite{barbour1992poisson}. The local version  
stated below (Theorem 4.7 from~\cite{ross2011fundamentals}) was instrumental in~\cite{gl1} and continues to play a similar role in the present paper. 
%In~\cite{gl1}, the main tool we use to prove Poisson convergence is the Chen-Stein method~\cite{barbour1992poisson}.
%This theorem is also used in this paper to prove Theorem~\ref{theorem: general label deciding probability}.
%So we briefly mention the result.

Let $\Poi(\lambda)$ be a Poisson random variable with expectation $\lambda$, and let $d_{\text{TV}}$ be the total variation distance
between measures on $\Z_+$.
%We need the following setting for our purposes.
Assume that $I_i$, $i\in \Gamma$, are indicators of a finite family of events, $p_i = \mathbb{E}(I_i)$, $\displaystyle W = \sum_{i\in \Gamma}I_i$, $\displaystyle \lambda = \sum_{i\in \Gamma} p_i = \mathbb{E}W$, and $\Gamma_i = \{j\in \Gamma: j \neq i,\, I_i\text{ and } I_j \text{ are not independent}\}$. 
 
\begin{lem} \label{lemma: chen-stein method}
We have
$$d_{\text{TV}}\left(W, \Poi\left(\lambda\right)\right) 
\le \min\left(1, \lambda^{-1}\right) \left[ \sum_{i\in \Gamma} p_i^2 
+ \sum_{i\in \Gamma, j\in \Gamma_i} \left(p_ip_j + \mathbb{E}\left(I_iI_j\right)\right)\right].$$
\end{lem}

\iffalse
In our applications of the above lemma, all deterministic and random quantities depend on the number $n$ of states, which we make explicit by the subscripts.
In our setting, we prove that $d_{\text{TV}}\left(W_n, \Poi\left(\lambda_n\right)\right) = \cO(1/n)$ and that $\lambda_n \to \lambda$ as $n \to \infty$, for an explicitly given $\lambda$, which implies that $W_n$ converges to $\Poi(\lambda)$ in distribution.
See Theorem~\ref{THEOREM: MAIN 2_1} and Theorem~\ref{THEOREM: MAIN 2_2}.
\fi

%\note{Up to here.}
 
\subsection{Simple Tiles} \label{subsection: simple tiles}
If a tile $T$ has zero lag, we call $T$ \textbf{simple}.
In~\cite{gl1}, we show that the probability of existence of PS with simple tiles provides the dominant terms of the existence of PS.
In Section~\ref{section: proof of theorem}, we show that it is also the dominant term for WRPS. %The next five lemmas are from~\cite{gl1}.
\iffalse
For example, consider the tiles
$$T_1 = 
\begin{matrix}
0 & 1 & 2 & 3\\
2 & 3 & 0 & 1\\
\end{matrix},\qquad
T_2 = 
\begin{matrix}
0 & 1 & 2 & 1\\
2 & 1 & 0 & 1\\
\end{matrix}.
$$
Then $T_1$ is simple, as $s(T_1) = p(T_1) = 4$, 
but $T_2$ is not, as $s(T_2) = 3$ and $p(T_2) = 4$.
Naturally, we call a PS simple if its tile is simple. 
\fi

\begin{lem} \label{lemma: simple tiles properties}   
Assume $T = (a_{i,j})_{ i = 0,\dots, \tau-1, j = 0, \dots, \sigma - 1}$ is a simple tile.
Then
\begin{enumerate}
\item the states on each row of $T$ are distinct;
\item if two rows of $T$ share a state, then they are circular shifts of each other;
\item the states on each column of $T$ are distinct; and
\item if two columns of $T$ share a state, then they are circular shifts of each other.
\end{enumerate}
\end{lem}

\iffalse
\begin{proof}
See~\cite{gl1}.

\noindent{\it Part 1\/}: When $\sigma = 1$, each row contains only one state, making the claim trivial.
Now, assume that $\sigma \ge 2$ and that $a_{i, j} = a_{i, k}$ for some $i$ and $j \neq k$.
We must have $a_{i, j + 1} = a_{i, k + 1}$ in order to avoid $p(T) > s(T)$. 
Repeating this procedure for the remaining states on $\mathtt{row}_i$ shows that this row is periodic, contradicting part 2 of Lemma~\ref{lemma: properties of tile}. 
 
\noindent{\it Part 2\/}: If $a_{i,j} = a_{k, m}$, for $i \neq k$, then the states to their right must agree, i.e., $a_{i, j + 1} = a_{k, m + 1}$, in order to avoid $p(T) > s(T)$. 
Repeating this observation for the remaining states on $\mathtt{row}_i$ and $\mathtt{row}_k$ gives the desired result. 
  
\noindent{\it Part 3\/}:
Assume a column contains repeated state, say $a_{i, j} = a_{k, j}$ for some $i, j$ and $k$. 
By part 2, $\mathtt{row}_i$ is exactly the same as $\mathtt{row}_k$, so that the temporal period of this tile can be reduced, a contradiction.
 
\noindent{\it Part 4\/}:
Assume that $a_{i, j} = a_{k, m}$, for $j \neq m$. 
Then $a_{i, j + 1} = a_{k, m + 1}$ by parts 1 and 2. 
So, $a_{i + 1, j + 1} = a_{k + 1, m + 1}$ by part 1 in Lemma~\ref{lemma: properties of tile}.
So, $a_{i + 1, j} = a_{k + 1, m}$, again by  parts 1 and 2.
Now, repeating the previous step for $a_{i + 1, j} = a_{k + 1, m}$ gives the desired result.
\end{proof}
\fi
\iffalse
We revisit the remark following Lemma~\ref{lemma: properties of tile}:  a tile may have periodic columns, but such a tile cannot be simple.   
\fi
Let $T = (a_{i,j})_{ i = 0, \dots, \tau-1, j = 0, \dots ,\sigma - 1}$ be a simple tile.
Let 
$$i = \min \{k = 1, 2, \dots, \tau - 1: \mathtt{row}_k = \pi(\mathtt{row}_0), \text{ for some circular shift } \pi: \Z^\sigma \to \Z^\sigma\}$$
be the smallest $i$ such that $\mathtt{row}_i$ is a circular shift of $\mathtt{row}_0$, and let $i = 0$ if and only if $T$ does not have circular shifts of $\mathtt{row}_0$ other than this row itself.
Then this circular shift satisfies $\mathtt{row}_{(j + i) \mod \tau} = \pi(\mathtt{row}_j)$, for all $j = 0, \dots, \tau - 1$ and $i$ is determined by the tile $T$; we denote this circular shift by $\pi_T^r$.
We denote by $\pi_T^c$ the analogous circular shift for columns.

\begin{lem} 
Let $T$ be a simple tile of a PS, and let $d_1 = \text{ord } (\pi_T^r)$ and $d_2 = \text{ord } (\pi_T^c)$.
Then $d_1$ and $d_2$ are equal and divide  $\gcd(\tau, \sigma)$.
\end{lem}

\iffalse
\begin{proof}
See~\cite{gl1}.

Fix an element as $a_{0, 0}$.
By Lemma~\ref{lemma: simple tiles properties}, parts 1 and 2, $a_{0, 0}$ appears in $d_1$ rows of $T$.
It also appears in $d_2$ columns by Lemma~\ref{lemma: simple tiles properties}, parts 3 and 4.
As a consequence, $d_1 = d_2$.
The divisibility follows from Lemma~\ref{lemma: euler totient}.

\end{proof}
\fi

\begin{lem} \label{lemma: value of s(T)}
An integer $s\le n$ is the number of states in a simple tile $T$ of PS if and only if there exists $d \mid \gcd (\tau, \sigma)$, such that $s = \tau\sigma/d$.  
\end{lem}

\iffalse
\begin{proof}
See~\cite{gl1}.

Let $T = (a_{i, j})_{ i = 0,\dots, \tau - 1, j = 0, \dots , \sigma - 1}$.
Assume that $s(T) = s$ and let $d = \text{ord }(\pi_T^r)$.
Then by Lemma~\ref{lemma: simple tiles properties}, parts 1 and 2, the first $\tau/d$ rows of $T$ contain all states that are in $T$.
As a result, $s = \tau\sigma/d$ and $d = \text{ord }(\pi_T^r) \mid \gcd (\tau, \sigma)$.

Now assume that $d \mid \gcd (\tau, \sigma)$.
Then there exists a circular shift $\pi: \Z^\sigma \to \Z^\sigma$, such that $\text{ord }(\pi) = d$.
To form a simple tile $T$ with $s(T) = \tau\sigma/d$ states, construct a rectangle of $\tau/d$ rows and $\sigma$ columns using $\tau\sigma/d$ different states in the first $\tau/d$ rows of $T$.
Let $\mathtt{row}_{\tau/d}$ be defined by $\pi (\mathtt{row}_{0})$ and the subsequent rows are all automatically defined by the maps that are assigned in the first $\tau/d$ rows, by Lemma~\ref{lemma: properties of tile}, part 1.
\end{proof}\fi

The above lemma gives the possible values of $s(T)$ for a simple tile $T$ and the next one enumerates the number of simple tiles of PS containing $s$ different states.

\begin{lem} \label{lemma: counting simple tile}
The number of simple tiles of PS with temporal periods $\tau$ and spatial period $\sigma$ containing $s$ states is $\displaystyle \varphi(d)\binom{n}{s}(s - 1)!$, where $d = \tau \sigma/s$.
\end{lem}
\iffalse
\begin{proof}
See~\cite{gl1}.

As in the proof of Lemma~\ref{lemma: value of s(T)}, if $s(T) = s = \tau\sigma/d$, then $d = \text{ord} (\pi_T^r)$.
Moreover, there are $\displaystyle \binom{n}{s} (s - 1)!$ ways to form the first $\tau/d$ rows of $T$.
Then, to uniquely determine $T$, we need to select a circular shift $\pi: \Z^\sigma \to \Z^\sigma$ with $\text{ord }(\pi) = d$ and define $\mathtt{row}_{\tau/d}$ to be $\pi (\mathtt{row}_{0})$.
By Lemma~\ref{lemma: euler totient}, there are $\varphi(d)$ ways to do so.
\end{proof}\fi

Consider two different simple tiles $T_1$ and $T_2$ under the rule.
The following lemma provides a lower bound on the combined number of values of the rule $f$ assigned by $T_1$ \textit{and} $T_2$, in terms of the number of states.
\iffalse
If $s(T_1) = s_1$, then $p(T_1) \ge s_1$, i.e., there are at least $s_1$ values assigned by $T_1$.
If there are $s_2^\prime$ states in $T_2$ that are not in $T_1$, then there are at least $s_2^\prime$ additional values to assign.
Therefore, a lower bound of the number of values to be assigned in $T_1$ and $T_2$ is $s_1 + s_2^\prime$.
The next lemma states that we can increase this lower bound by at least $1$ when $T_1 \cap T_2 \neq \emptyset$.
This fact plays an important role in the proofs of Theorem~\ref{THEOREM: MAIN 2_1} and Theorem~\ref{THEOREM: MAIN 2_2}.
\fi
\begin{lem} \label{lemma: one more map}
Let $T_1$ and $T_2$ be two different simple tiles 
for the same rule.
If $T_1$ and $T_2$ have at least one state in common, then there exist $a_{i, j} \in T_1$ and $b_{k, m} \in T_2$ such that $a_{i, j} = b_{k, m}$ and $a_{i, j + 1} \neq b_{k, m + 1}$.
\end{lem}

\iffalse
\begin{proof}
See~\cite{gl1}.

As $T_1$ and $T_2$ have at least one state in common, we may pick $a_{i, j} \in T_1$ and $b_{k, m} \in T_2$, such that $a_{i, j} = b_{k, m}$.
If $a_{i, j + 1} \neq b_{k, m + 1}$, then we are done.
Otherwise, we repeat this procedure for $a_{i, j + 1}$ and $b_{k, m + 1}$ and see if $a_{i, j + 2} = b_{k, m + 2}$.
We repeat this procedure until we find two pairs such that $a_{i, j + q} = b_{k, m + q}$ and $a_{i, j + q + 1} \neq b_{k, m + q + 1}$.
If we fail to do so, then $\mathtt{row}_i$ in $T_1$ and $\mathtt{row}_k$ in $T_2$ must be equal, up to a circular shift.
This implies that $T_1$ and $T_2$ must be the same since they are tiles for same rule, a contradiction.
\end{proof}
\fi

As a result, if $s(T_1) = s_1$, then $p(T_1) \ge s_1$, i.e., there are at least $s_1$ values assigned by $T_1$.
If there are $s_2^\prime$ states in $T_2$ that are not in $T_1$, then there are at least $s_2^\prime$ additional values to assign.
With the above lemma, a lower bound of the number of values to be assigned in $T_1$ and $T_2$ is $s_1 + s_2^\prime + 1$.

\section{Decidability and WRPS} \label{section: decidability and wrps}
%Recall the definition of right-extension, label digraph and Algorithm~\ref{algorithm: ps from label digraph} for obtaining PS from label digraph.
In order for a PS to be weakly robust, we need one more condition on the directed cycle in the label digraph, which requires that each label \textit{decides} its unique child. 
To be more accurate, let $A$ and $B$ be two labels. Assume that 
at a site $k\in \Z$ the temporal evolution 
of the states, arranged vertically, is the repeated label $A$: $a_0\dots a_{\tau - 1}a_0\dots a_{\tau - 1}\dots$.  Suppose 
that the states at site $k + 1$ eventually ``converge'' to repetition of $B$: $b_0\dots b_{\tau - 1}b_0\dots b_{\tau - 1}\dots$, regardless of the initial state at site $k + 1$. In this 
case, we say that $A$ decides $B$, and then it is clear that $A$ does not decide $C$ for any other length-$\tau$ label $C$ that is not equal to $B$ up to a circular shift. We now provide a more formal definition. 

\begin{defi} \label{definition: decidability}
Let $A = a_0\dots a_{\tau - 1}$ and $B = b_0\dots b_{\tau - 1}$ be two length-$\tau$ labels. 
We call that label $A$ \textbf{decides} $B$, denoted as $A\Rightarrow B$, if the following two conditions are satisfied:
\begin{enumerate}
\item label $A$ right-extends to $B$, i.e., $A\to B$;

\item for an arbitrary $c_0\in \mathbb{Z}_n$, recursively define $c_{j + 1} = f(a_{j\mod \tau}, c_j)$; then there exists a $j \ge 0$ such that $c_{j \mod \tau} = b_{j \mod \tau}$.
\end{enumerate}
\end{defi}

The following proposition, analogous to Proposition 2.2 in~\cite{gravner2012robust}, provides an algorithm to verify whether a PS is weakly robust.

\begin{prop}
A tile is a WRPS if and only if each column decides the column to its right.
\end{prop}

\begin{proof}
Assume that a tile $T = (a_{i, j})$ is a WRPS with columns $A_j$, 
$j = 0, \dots, \sigma - 1$. 
Let $\eta$ be the initial configuration formed by doubly infinite repetition of $a_{0, 0}\dots a_{0, \sigma - 1}$.  
If $A_j = a_{0, j}\dots a_{\tau - 1, j}$ does not decide $A_{j + 1} = a_{0, j + 1}\dots a_{\tau - 1, j + 1}$, for some $j = 0, \dots, \tau - 1$, then there exists a $c_0 \in \Z_n$ such that in the position to the right of $A_j$, the states do not converge to a repetition of
$A_{j + 1}$. 
Now, construct an initial configuration  $\eta^\prime$ 
by replacing one $a_{0, j + 1}$ by $c_0$ in $\eta$. 
Then $\eta^\prime$ is proper for $\eta$, but the advance of the spatial period is stopped, thus $v(\eta^\prime) = 0$ and $T$ cannot be weakly robust.

Conversely, note that if label $A_j$ decides $A_{j + 1}$, then for any $c_0 \in \Z_n$ to the right of $a_{0, j}$, the label converges to $A_{j + 1}$ within $n\tau$ iterations.
Thus the expansion velocity must be at least $1/(\tau n)$.
\end{proof}

Recall that by Lemma~\ref{lemma: properties of tile}, a tile of a PS 
does not have periodic rows.
The following lemma concludes that a periodic label cannot be a part of WRPS tile, since otherwise the temporal period of the WRPS is reduced. 

\begin{lem} \label{lemma: irreducibility of labels}
If $T$ is a tile of WRPS of period $\tau$, then every column has minimal period $\tau$.
\end{lem}
\begin{proof}
Assume that $A$ is a label of length $\tau$ that is formed by concatenating shorter label $A^\prime$ that has length $\tau^\prime$.
It is clear that if $A \Rightarrow B = b_0 \dots b_{\tau - 1}$, $A$ also decides the circular shift $b_{\tau^\prime} b_{\tau^\prime - 1} \dots b_{\tau}b_{0} \dots b_{\tau^\prime - 1}$.
This implies that $b_0 = b_{\tau^\prime}$, $b_1 = b_{\tau^\prime + 1}$, etc.
That is, $B$ is also periodic with period $\tau^\prime$.
By induction, every label in $T$ is periodic with period $\tau^\prime$, thus $T$ is temporally reducible.
\end{proof}

In a label digraph $D_{\tau, f}$, we call an arc $A\to B$ \textbf{deciding arc} if $A\Rightarrow B$ and a directed cycle \textbf{deciding cycle} if all the arcs contained in this cycle are deciding arcs. 
The following algorithm finds all WRPS of temporal period $\tau$ for rule $f$.

\begin{alg}\label{algorithm: wrps from label digraph}
\quad \\
\noindent\rule{\textwidth}{1pt}
\begin{algorithm}[H] %\label{algorithm: generating random DEC}
\SetKwInOut{Input}{input}
\SetKwInOut{Output}{output}
\SetKw{Print}{print}
\SetKw{Return}{return}
\DontPrintSemicolon

\Input{
Label digraph $D_{\tau, f}$ of $f$ with temporal period $\tau$\;
}
\;
Find all deciding cycles in $D_{\tau, f}$\;
\For{each deciding cycle $A_0 \Rightarrow A_1 \Rightarrow \dots \Rightarrow A_{\sigma - 1}\Rightarrow A_0$}{form the tile $T$   
by placing labels $A_0, A_1,\ldots,  A_{\sigma - 1}$ on successive columns.\;
	\If{both spatial and temporal periods of $T$  are minimal}{
		\Print $T$ as a WRPS\;
	}
}
\end{algorithm}

\noindent\rule{\textwidth}{1pt}
\end{alg}

\section{Decidability Probability} \label{section: decidability probability}
We call a label $A= a_0\dots a_{\tau - 1}$ \textbf{simple} if $a_i \neq a_j$ for $i \neq j$.
We next prove the main result regarding the probability of the decidability of simple labels.

\begin{thm} \label{theorem: simple label deciding probability}
Fix a number of states $n$ and a $\tau \le n$. 
Let $A= a_0\dots a_{\tau - 1}$ be a simple label with length $\tau$ and $B=b_0\dots b_{\tau - 1}$ be any other label (not necessarily simple) of length $\tau$. 
Then 
$$\mathbb{P}\left(A \Rightarrow B\right) = \frac{n^\tau - (n - 1)^\tau}{n^\tau} \cdot \frac{1}{n^\tau}.$$
\end{thm}

The theorem is proved in four lemmas below.
The key idea reduces to calculating the probability that a random $\tau$-partite graph is a directed \textbf{pseudo-tree}, i.e., a weakly connected directed graph that has at most one directed cycle. 
To be precise, we construct \textbf{label assignment digraph} (LAD) $G_{\tau, n}(f, A)$ of a label $A$ under a rule $f$ in the following manner.

We consider $\tau$-partite digraphs with the $i$th part denoted
by $(i,\ast)=\{(i,j): j=0, \ldots, n-1\}$, $i=0,\ldots,\tau-1$. 
The arcs of the digraph $G_{\tau, n}(f, A)$ 
are determined as follows: 
for all $i = 0, \dots, \tau - 1$ and $j = 0, \dots, n - 1$, 
there is an arc $(i, j) \to (i + 1, j^\prime)$ if $f(a_{i}, j) = j^\prime $. 
As usual, we identify $i=\tau$ with $i=0$, $i={\tau + 1}$ with $i=1$, etc.
We next state the conditions for $G_{\tau, n}(f, A)$ that characterize when $A\to B$ and when $A \Rightarrow B$.
 
%
%The graph $G_{\tau, n}(f, A)$ is a $\tau$-partite graph and we name the parts of $G_{\tau, n}(f, A)$ as $(i, \ast)$, for $i = 0, \dots, \tau - 1$.
%Denote $(i, j)$ as the $j$th node in the $i$th part for $j = 0, \dots, n - 1$.
%Construct arcs $(i, j) \to (i + 1, j^\prime)$ if $f(a_{i}, j) = j^\prime $, for all $i = 0, \dots, \tau - 1$ and $j = 0, \dots, n - 1$.
%As usual, we identify $a_\tau$ with $a_0$, $a_{\tau + 1}$ with $a_1$, etc.
%We next state the conditions for $G_{\tau, n}(f, A)$, which  characterize when $A\to B$ and when $A \Rightarrow B$.

\begin{defi}
Let $A = a_0\dots a_{\tau - 1}$ and $B = b_0\dots b_{\tau - 1}$ be two labels. 
Consider the following conditions on a $\tau$-partite graph $G$:
\begin{enumerate}
\item $G$ contains the cycle $(0, b_0) \to (1, b_1) \to \dots \to (\tau - 1, b_{\tau - 1}) \to (0, b_0)$;
\item there is a directed path in $G$ from $(i, j)$ to $(0, b_0)$ for all $i = 0, \dots, \tau - 1$ and $j = 0, \dots, n - 1$. 
\end{enumerate}
The set $\mathcal{E}(A, B)$ is the set of all $\tau$-partite digraphs $G$, which satisfy condition (1) and the set $\mathcal{D}(A, B)$ is the set of all such digraphs $G$ that satisfy both conditions (1) and (2).
\end{defi}

\begin{lem} \label{lemma: lad}
Let $A = a_0\dots a_{\tau - 1}$ and $B = b_0\dots b_{\tau - 1}$ be any two labels. 
Then $A\to B$ if and only if $G_{\tau, n}(f, A)\in \mathcal{E}(A, B)$ and $A \Rightarrow B$ if and only if $G_{\tau, n}(f, A)\in \mathcal{D}(A, B)$.
\end{lem}

We skip the proof as it follows immediately from the definitions, and instead give two examples for different rules 
by Figure~\ref{figure: extend vs decide}.
For the reader's convenience, we denote a node $(i[a_i], j)$
instead of  $(i, j)$ as in the definition.
%We remark here that, while in the original definition, a node is denoted as $(i, j)$, in the given example, we denote a node as $(i[a_i], j)$ for the reader's convenience. 
The two labels
are $A = 12$ and $B = 00$ in both cases.
Under the rule that generates the left LAD, $A\to B$, but $A \not \Rightarrow B$, i.e., $G_{\tau, n}(f, A)\in \mathcal{E}(A, B) \setminus \mathcal{D}(A, B)$; under the rule that generates the right LAD, $A \Rightarrow B$, i.e., $G_{\tau, n}(f, A)\in \mathcal{D}(A, B)$.

\begin{figure}[ht]
\centering
\begin{tikzpicture}
[
            > = stealth, % arrow head style
            shorten > = 1pt, % don't touch arrow head to node
            auto,
            node distance = 3cm, % distance between nodes
            semithick % line style
        ]

        \tikzstyle{every state}=[
            draw = black,
            thick,
            fill = white,
            minimum size = 4mm
        ]
        
        \node (1) at (0, 0) {$(0[1], 0)$};
        \node (2) at (0, -1) {$(0[1], 1)$};
        \node (3) at (0, -2) {$(0[1], 2)$};
        \node (4) at (3, 0) {$(1[2], 0)$};
        \node (5) at (3, -1) {$(1[2], 1)$};
        \node (6) at (3, -2) {$(1[2], 2)$};
        
		\path[->] (1) edge [bend left = 10] node {} (4);
        \path[->] (4) edge [bend left = 10] node {} (1);
        \path[->] (2) edge node {} (5);
        \path[->] (5) edge node {} (3);
        \path[->] (3) edge node {} (6);
        \path[->] (6) edge node {} (2);
	        
\end{tikzpicture}
\quad 
\hspace{3cm}
\begin{tikzpicture}
[
            > = stealth, % arrow head style
            shorten > = 1pt, % don't touch arrow head to node
            auto,
            node distance = 3cm, % distance between nodes
            semithick % line style
        ]

        \tikzstyle{every state}=[
            draw = black,
            thick,
            fill = white,
            minimum size = 4mm
        ]
        
        \node (1) at (0, 0) {$(0[1], 0)$};
        \node (2) at (0, -1) {$(0[1], 1)$};
        \node (3) at (0, -2) {$(0[1], 2)$};
        \node (4) at (3, 0) {$(1[2], 0)$};
        \node (5) at (3, -1) {$(1[2], 1)$};
        \node (6) at (3, -2) {$(1[2], 2)$};
        
		\path[->] (1) edge [bend left = 10] node {} (4);
		\path[->] (4) edge [bend left = 10] node {} (1);
		\path[->] (2) edge node {} (4);
		\path[->] (5) edge node {} (3);
		\path[->] (3) edge node {} (6);
		\path[->] (6) edge node {} (2);
	        
\end{tikzpicture}
\caption{Two LADs of label $A = 12$ under two different rules. 
We use $(i [a_i], j)$ to represent a node for the reader's convenience.
In the left one, $A \to 00$ but $A \not \Rightarrow 00$; in the right one, $A \Rightarrow 00$.}
\label{figure: extend vs decide}

\end{figure}
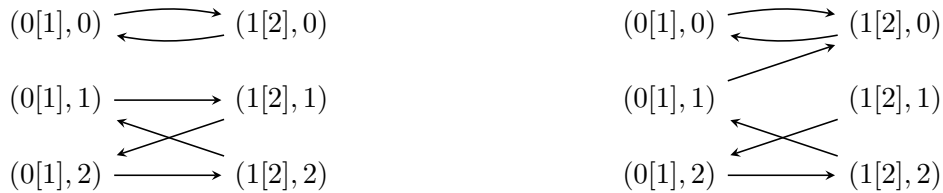

Fix a label $A = a_0\dots a_{\tau - 1}$. 
The LAD $G_{\tau, n} (f, A)$ becomes a random graph if the rule $f$ is selected randomly and we are interested in $\mathbb{P}\left(G_{\tau, n} (f, A) \in \mathcal{E}(A, B)\right)$ and $\mathbb{P}\left(G_{\tau, n} (f, A) \in \mathcal{D}(A, B)\right)$.
The case that $A$ is simple is easier as we can take advantage of independence of assignments of $f$.
To be precise, let $A$ be a simple label with length $\tau$ and $B$ be an arbitrary label with the same length.
We clearly have that $\mathbb{P}\left(G_{\tau, n}(f, A) \in \mathcal{E}(A, B)\right) = 1/n^\tau$, as the assignments on $(a_j, b_j)$'s are independent.

Next, we find $\mathbb{P}\left(G_{\tau, n} (f, A) \in \mathcal{D}(A, B)\right)$ for simple label $A$ thus complete the proof of Theorem~\ref{theorem: simple label deciding probability}.
We start by the following observation.

\begin{lem} \label{lemma: label does not matter}
If $A$ and $A^\prime$ are simple labels with the same length, 
$\mathbb{P}(A \Rightarrow B) = \mathbb{P}(A^\prime \Rightarrow B)$ for any label $B$; if $B$ and $B^\prime$ are labels with the same length, $\mathbb{P}(A \Rightarrow B) = \mathbb{P}(A\Rightarrow B^\prime )$ for any simple label $A$.
\end{lem}

To find $\mathbb{P}\left(G_{\tau, n} (f, A) \in \mathcal{D}(A, B)\right)$, we adapt the counting techniques in~\cite{moh1990number} to enumerate $\mathcal{D}(A, B)$.
We start by proving the following combinatorial result.
\iffalse
\begin{lem} \label{lemma: combinatorics identity 1}
For any $x, y > 0 $, $\displaystyle \sum_{k=0}^n \binom{n}{k}x^ky^{n - k}k = xn(x + y)^{n - 1}$.
\end{lem}

\begin{proof}
Note that 
\begin{align*}
\sum_{k=0}^n \binom{n}{k} x^ky^{n - k}k 
& = (x + y)^n\sum_{k = 0}^n \binom{n}{k} \left(\frac{x}{x + y }\right)^{n - k}\left(\frac{y}{x + y}\right)^{k} k\\
& = (x + y)^n \mathbb{E}Y,
\end{align*}
where $Y \sim \Bin(n, x/(x+y))$.
\end{proof}
\fi
\begin{lem} \label{lemma: combinatorics identity}
Let $\displaystyle A_{k, \ell} = \binom{n - 1}{k}(\ell + 1)^{k}(n - 1 - \ell)^{n - 1 - k}$ , and assume that 
$k_{m + 1}$ is a non-negative integer.
Then
\begin{align*}
S_m 
& := \sum_{k_m = 0}^{n - 1} A_{k_m, k_{m + 1}}\dots \left[\sum_{k_2 = 0}^{n - 1}A_{k_2, k_3}\left[\sum_{k_1 = 0}^{n - 1}A_{k_1, k_2}(k_1 + 1)n^{n - 2}\right]\right] \\
& = n^{(m + 1)(n - 2)}\left[P_{m + 1} + k_{m + 1}(n - 1)^m \right],
\end{align*}
where
$P_m = n^m - (n - 1)^m$.
\begin{proof}
We use induction on $m$. 
Assume $m = 1$.
Observe that $$A_{k, \ell} = n^{n - 1} \mathbb{P}\left(\Bin\left(n - 1, \frac{\ell + 1}{n}\right) = k\right).$$
Therefore, 
\begin{align*}
\sum_{k_1 = 0}^{n - 1} A_{k_1, k_2}(k_1 + 1)n^{n - 2} & = n^{n - 2} \cdot n^{n - 1} \cdot \left[1 + (n - 1) \frac{k_2 + 1}{n}\right]\\
& = n^{2(n - 2)}\left[P_2 + k_2(n - 1)\right].
\end{align*}
\iffalse
\begin{align*}
& \quad \sum_{k_1 = 0}^{n - 1} \binom{n - 1}{k_1}(k_{2} + 1)^{k_1}(n - 1 - k_{2})^{n - 1 - k_1}(k_1 + 1)n^{n - 2}\\
& = n^{n - 2} \left[\sum_{k_1 = 0}^{n - 1}\binom{n - 1}{k_1}(k_{2} + 1)^{k_1}(n - 1 - k_{2})^{n - 1 - k_1}k_1 \right.\\
& \quad \left. + \sum_{k_1 = 0}^{n - 1}\binom{n - 1}{k_1}(k_{2} + 1)^{k_1}(n - 1 - k_{2})^{n - 1 - k_1}\right]\\
& = n^{n - 2}\left[(k_2 + 1)(n - 1)n^{n - 2} + n^{n - 1}\right] \\
& = n^{2(n - 2)}\left[P_2 + k_2(n - 1)\right],
\end{align*}
where the second equality follows from Lemma~\ref{lemma: combinatorics identity 1}. 
\fi
Now, by the induction hypothesis
\begin{align*}
S_m &= \sum_{k_m=0}^{n - 1} A_{k_m, k_{m + 1}}S_{m - 1}\\
&= n^{m(n - 2)}\sum_{k_m=0}^{n - 1} \binom{n - 1}{k_m}(k_{m + 1} + 1)^{k_m}(n - 1 - k_{m + 1})^{n - 1 - k_m}\left[P_m + k_{m}(n - 1)^{m - 1} \right]\\
&= n^{m(n - 2)}\left[n^{n - 1}P_m + (n - 1)^{m}(k_{m + 1} + 1)n^{n - 2}\right]\\
&= n^{(m + 1)(n - 2)}\left[nP_m + k_{m + 1}(n - 1)^{m} + (n - 1)^{m}\right]\\
&= n^{(m + 1)(n - 2)}\left[P_{m + 1} + k_{m + 1}(n - 1)^m\right],
\end{align*}
which is the desired result.
\end{proof}
\end{lem}

Now, we are ready to prove the key combinatorial result.

\begin{lem} \label{lemma: mapping graph deciding probability}
Let $A$ and $B$ be labels with length $\tau$ and let $A$ be simple. 
Then $\#\mathcal{D}(A, B) = n^{\tau(n - 2)}(n^\tau - (n - 1)^\tau)$.
\end{lem}
\begin{proof}\
The argument we give partly follows the proof of Theorem 1 in 
\cite{moh1990number}.
Applying Lemma~\ref{lemma: label does not matter}, we may assume that $B = 0\dots 0$, without loss of generality.

First, choose a $k_{\tau - 1} \in \{0, \dots, n - 1\}$, pick $k_{\tau - 1}$ nodes in $(\tau - 1, \ast) \setminus \{(\tau - 1, 0)\}$, and form $k_{\tau - 1}$ arcs from those nodes to the node $(0, 0)$. 
There are $\displaystyle \binom{n - 1}{k_{\tau - 1}}$ choices for a fixed $k_{\tau - 1}$. 
Denote this subset of $(\tau - 1, \ast)$ together with $(\tau - 1,0)$ as $(\tau - 1, \ast)^\prime$; thus, $(\tau - 1, \ast)^\prime\subset (\tau - 1, \ast)$ are the nodes in $(\tau - 1, \ast)$ that are mapped to $(0, 0)$. 
Assign the images of the nodes in $(\tau - 1, \ast)\setminus(\tau - 1, \ast)^\prime$ to $(0, \ast) \setminus \{(0, 0)\}$, for which there are  
$(n - 1)^{n - 1 - k_{\tau - 1}}$ choices. 
So, for a fixed $k_{\tau - 1}$ to assign the image of nodes in $(\tau - 1, \ast)$, there are 
$$\binom{n - 1}{k_{\tau - 1}}(n - 1)^{n - 1 - k_{\tau - 1}}$$ 
choices.

Second, we need to assign the image of the nodes in $(\tau - 2, \ast)$ to $(\tau - 1, \ast)$. Choose a $k_{\tau - 2}\in \{0, \dots, n - 1\}$, 
pick $k_{\tau - 2}$ nodes in $(\tau - 2, \ast) \setminus {(\tau - 2,0)}$, and form $k_{\tau - 2}$ arcs from those nodes to the nodes in $(\tau - 1, 0)^\prime$. 
There are $\displaystyle \binom{n - 1}{k_{\tau - 2}}$ choices to choose those nodes for a fixed $k_{\tau - 2}$ and $(k_{\tau - 1} + 1)^{k_{\tau - 2}}$ choices to assign the images. 
Denote this subset of $(\tau - 2, \ast)$ together with $(\tau - 2,0)$ as $(\tau - 2, \ast)^\prime$. Now, the images of the nodes in $(\tau - 2, \ast) \setminus (\tau - 2, \ast)^\prime$ should be in $(\tau - 1, \ast) \setminus (\tau - 1, \ast)^\prime$, for 
which there are $\displaystyle (n - 1 - k_{\tau - 1})^{n - 1 - k_{\tau - 2}}$ choices. 
Hence, for fixed $k_{\tau-1}$ and $k_{\tau-2}$, to assign the image of the nodes in $(\tau - 2,\ast)$ to $(\tau - 1, \ast)$, there are
$$\binom{n - 1}{k_{\tau - 2}} (k_{\tau - 1} + 1)^{k_{\tau - 2}}(n - 1 - k_{\tau - 1})^{n - 1 - k_{\tau - 2}} $$
choices.

Repeat the above steps for $(\tau - 3, \ast)$, \ldots, $(1, \ast)$. 
To complete the construction, we assign the images of the nodes in $(0, \ast) \setminus \{(0, 0)\}$. 
We choose a $t \in \{0,\dots, n - 2\}$, and add $t$  
arcs from $(0, \ast) \setminus \{(0, 0)\}$ to $(1, \ast) \setminus (1, \ast)^\prime$ consecutively as specified below, 
making sure to avoid creating a cycle that does not include 
$(0,0)$.  

In the evolving digraph, a {\bf component} 
is a weakly connected component, obtained by ignoring the orientation of edges. 
First note that there are $n$ components in the current digraph; more precisely, each node of $(0,\ast)$ belongs to a
different component (possibly consisting of a single node). 

To select the first arc, pick a $b \in (1, \ast) \setminus (1, \ast)^\prime$ ($n - 1 - k_1$ choices).
There is one component that contains $(0, 0)$ and one other component containing $b$.
As a result, there are $n - 2$ components and among each of them, there is a node in  $(0, \ast) \setminus \{(0, 0)\}$ with zero out-degree. Among these $n-2$ nodes,
we select one and connect it to $b$.
Therefore, there are $(n - 2)(n - 1 - k_1)$ choices for the first arc. The addition of this arc
decreases the number of components by one.

To assign the second arc, again pick a $b \in (1, \ast) \setminus (1, \ast)^\prime$ (again $n - 1 - k_1$ choices).
Now there are exactly $n - 3$ components, among which there is a node  in $(0, \ast) \setminus \{(0, 0)\}$  with zero out-degree.
We again select one and connect it with this $b$, leading to $(n - 3)(n - 1 - k_1)$ choices.

In subsequent steps, we add an arc from $a$ to $b$, where 
$b\in (1, \ast) \setminus (1, \ast)^\prime$ is 
arbitrary, while $a\in (0, \ast) \setminus \{(0, 0)\}$  is a unique node with zero out-degree in any component not containing $b$ in the graph already constructed.
The algorithm guarantees that the 
number of components decreases by one after each arc is added, 
i.e., that a cycle not including $(0,0)$ is never created.

In the above steps we add $t$ arcs, with the number of 
choices, in order: $(n - 2)(n - 1 - k_1), (n - 3)(n - 1 - k_1)\ldots, (n - t - 1)(n - 1 - k_1)$. As any order in which they 
are assigned produces the same digraph, there are 
$$\frac{(n - 2)(n - 1 - k_1)(n - 3)(n - 1 - k_1)\cdots (n - t - 1)(n - 1 - k_1)}{t!} = \binom{n - 2}{t} (n - 1 - k_1)^t$$
choices. 
Finally, we assign the remaining $n - 1 - t$ arcs to $(1, \ast)^\prime$, for which we have $(k_1 + 1)^{n - 1 - t}$ choices. Hence, for a fixed $k_1$, to assign the arcs originating from  $(0, \ast)\setminus \{(0, 0)\}$, there are
$$\sum_{t = 0}^{n - 2}\binom{n - 2}{t}(n - 1 - k_1)^t(k_1 + 1)^{n - 1 - t}=(k_1 + 1)n^{n - 2}$$
choices, in total.
Lastly, we use Lemma~\ref{lemma: combinatorics identity}
to get
\begin{align*}
\#\mathcal{D}(A, B) &= \sum_{k_{\tau - 1}=0}^{n - 1}\binom{n - 1}{k_{\tau - 1}}(n - 1)^{n - 1 - k_{\tau - 1}}\\
& \quad  \cdot \left[\sum_{k_{\tau - 2}=0}^{n - 1} A_{k_{\tau - 2},k_{\tau - 1}}\cdots \left[\sum_{k_2=0}^{n - 1}A_{k_2, k_3}\left[\sum_{k_1=0}^{n - 1}A_{k_1, k_2}(k_1 + 1)n^{n - 2}\right]\right]\cdots \right]\\
& = n^{(\tau - 1)(n - 2)}\sum_{k_{\tau - 1}=0}^{n - 1}\binom{n - 1}{k_{\tau - 1}} (n - 1)^{n - 1 - k_{\tau - 1}}[P_{\tau - 1} + k_{\tau - 1}(n - 1)^{\tau - 2}]\\
& = n^{(\tau - 1)(n - 2)} \left[n^{n - 1}P_{\tau - 1} + (n - 1)^{\tau - 1}n^{n - 2}\right]\\
& = n^{\tau(n - 2)} P_{\tau}, 
\end{align*} 
as claimed.
\end{proof}

Now, proof of Theorem~\ref{theorem: simple label deciding probability} is straightforward.

\begin{proof}[Proof of Theorem~\ref{theorem: simple label deciding probability}]
It is clear that the number of LAD $G_{\tau, n}(f, A)$ is $n^{\tau n}$.
Then, by Lemma~\ref{lemma: mapping graph deciding probability},
$$\mathbb{P}(A \Rightarrow B) = \mathbb{P}(G_{\tau, n}(f, A) \in \mathcal{D}(A, B)) =  \frac{n^{\tau(n - 2)}[n^\tau - (n - 1)^\tau]}{n^{\tau n}} = \frac{n^\tau - (n - 1)^\tau}{n^\tau} \cdot \frac{1}{n^\tau},$$
as claimed.
\end{proof}

By Theorem~\ref{theorem: simple label deciding probability},
assuming that $A$ is simple and $B$ is any label of the same length $\tau$, we have
$$\mathbb{P}(A \Rightarrow B \bigm| A \to B) = \frac{n^\tau - (n - 1)^\tau}{n^\tau} = \frac{\tau}{n} + o\left(\frac{1}{n}\right).$$

The case when $A$ is not simple is much harder, since the parts of $G_{\tau, n}(f, A)$ are no longer independent from each other for a random rule $f$. 
While it is possible to obtain the deciding probability for a specific label using a similar method as in Theorem~\ref{theorem: simple label deciding probability}, it is hard to find a general formula or even to prove this probability is 
always $\mathcal{O}(1/n)$. 
We are, however, able to obtain the following weaker result.

\begin{thm} \label{theorem: general label deciding probability}
Let $A = a_0\dots a_{\tau - 1}$ and $B = b_0\dots b_{\tau - 1}$ be two fixed labels (not necessarily simple) with length $\tau$. 
Then 
$$\mathbb{P}\left(G_{\tau, n}(f, A) \in \mathcal{D}(A, B) \bigm| G_{\tau, n}(f, A) \in \mathcal{E}(A, B)\right) = o(1).$$
Equivalently, we have
$$\mathbb{P}\left(A \Rightarrow B\bigm| A \to B\right) = o(1).$$
\end{thm}

\begin{proof}
Again, we assume that $B = 0\dots 0$. 
We remark that, unlike Theorem~\ref{theorem: simple label deciding probability}, label $B$ here does affect the deciding probability.
However, the case of general $B$ does not significantly alter the proof but it makes it transparent, so we choose this $B$ 
for readability.

Let $a_0^\prime, \dots, a_{\ell - 1}^\prime$ be the different states in $A$ and $m_i$ be the repetition numbers of $a_i$'s, for $i = 0,\dots, \ell - 1$. 
Clearly, $\displaystyle \sum_{i = 0}^{\ell - 1}m_i = \tau$.
Let $\Cy$ be the cycle $(0, 0) \to (1, 0) \to \dots \to (\tau - 1, 0) \to (0, 0)$.  
It suffices to show that $$\mathbb{P}(\text{there are no other cycles in }G_{\tau, n}(f, A)\bigm| \Cy \in  G_{\tau, n}(f, A)) = o(1).$$
To accommodate the conditional probability,  our probability space will be a uniform choice of a digraph from 
$\mathcal{E}(A, B)$ for the remainder of the proof. 
 
Fix an integer $K\ge 1$. 
Call a cycle $\Cy^\prime = (0, j_0) \to (1, j_1) \to \dots \to (0, j_0)$ \textbf{simple with respect to $\Cy$} if:
\begin{enumerate}
\item $\Cy^\prime$ contains no parallel arcs, i.e., if $(i, j)$ and $(i^\prime, j)$ are nodes in $\Cy^\prime$, then $a_i \neq a_{i^\prime}$; and
\item 
if $(i, j)$ is on $\Cy$ and $(i^\prime, j^\prime)$ on $\Cy^\prime$, then $(a_{i^\prime}, b_{j^\prime}) \neq (a_{i}, b_{j})$.
\end{enumerate}
Let $Y_k$ be the random number of simple cycles with respect to $\Cy$ with length exactly $\tau k$ and $\displaystyle Z_K = \sum_{k = 1}^K Y_k$ be the random variable that counts the number of such cycles with length less than or equal to $\tau K$. 
We will show that, for any $K$, $\displaystyle \lim_{n \to \infty} \mathbb{P}(Z_K \ge 1) = 1 - \exp\left(-\sum_{k = 1}^K 1/k\right)$, converging to 1 as $K \to \infty$.
As a consequence, the LAD has another simple cycle asymptotically almost surely (in $n$), and this will conclude the proof.

We first compute the expectation of $Y_k$:
$$\mathbb{E}Y_k = \frac{(n - 1)_{m_1k} \cdots (n - 1)_{m_\ell k}}{k}\cdot \frac{1}{n^{\tau k}} \to \frac{1}{k}, \quad \text{as } n \to \infty.$$
Here and in the sequel, we use the falling factorial notation $(x)_n = x(x - 1)\cdots (x - n + 1)$. The first factor counts the number of simple cycles with respect to $\Cy$ and the second factor is the probability that a  fixed simple cycle with length $\tau k$ is formed.

Now, let $\displaystyle \lambda_{K} = \mathbb{E}Z_K = \sum_{k = 1}^K\mathbb{E}Y_k$.
%We use the Stein-Chen method to bound the total variation distance $d_{TV}$ between $Z_K$ and $\Poi(\lambda_K)$, a Poisson random variable with expectation $\lambda_K$. 
We use the notation $\Gamma^k$ to denote the set of all possible simple cycles with length $\tau k$ and define $\displaystyle \Gamma = \bigcup_{1 \le k \le K} \Gamma^k$ as set of such cycles with length less than or equal to $\tau K$. 
The set $\Gamma_i$ consists of cycles in $\Gamma$ that has at least one node in common with the cycle $i$. 
The random variable $I_i$ is the indicator that the cycle $i \in \Gamma$ is formed and $p_i = \mathbb{E}I_i$.
%
%Recall that by the Chen-Stein lemma, i.e., Lemma~\ref{lemma: chen-stein method}, we have
%$$d_{\text{TV}}(Z_K, \Poi(\lambda_K)) \le \min(1, \lambda_K^{-1}) \left[ \sum_{i\in \Gamma} p_i^2 + \sum_{i\in \Gamma, j\in \Gamma_i} \left(p_ip_j + \mathbb{E}\left(I_iI_j\right)\right)\right].$$

We use Lemma~\ref{lemma: chen-stein method} to find an 
upper bound for
$d_{\text{TV}}(Z_K, \Poi(\lambda_K))$. 
For the first term $\displaystyle \sum_{i\in \Gamma}p_i^2$, we have
$$\sum_{i\in \Gamma}p_i^2 = \sum_{k = 1}^K\frac{(n - 1)_{m_1k} \cdots (n - 1)_{m_\ell k}}{k}\frac{1}{n^{2\tau k}} = \mathcal{O}\left(\frac{1}{n^\tau}\right).$$
To obtain an upper bound for $\displaystyle \sum_{i \in \Gamma}\sum_{j \in \Gamma_i}p_ip_j$, we note that if $i$ is the index of a simple cycle of length $\tau r$, then we may count the number of length-$\tau k$ simple cycles that have no common vertex with the cycle $i$, that is
$$\#\left(\Gamma^k \setminus \Gamma_i\right) = \frac{(n - 1 - r)_{m_1 k} \cdots (n - 1 - r)_{m_\ell k}}{k}.$$
It immediately follows that,
\begin{align*}&\#\left(\Gamma^k \cap \Gamma_i\right) \\&= \frac{(n - 1)_{m_1k} \cdots (n - 1)_{m_\ell k} - (n - 1 - r)_{m_1 k} \cdots (n - 1 - r)_{m_\ell k}}{k}
\\& = \mathcal{O}\left(n^{\tau k-1}\right), 
\end{align*}
as the highest powers of $n$ in the numerator cancel.
Hence, for a fixed $r$ and $k$, we have
\begin{align*}
& \sum_{i \in \Gamma^r} \sum_{k \in \Gamma_i \cap \Gamma^k} p_ip_j \\
%& = \frac{(n - 1)_{m_1 r} \cdots (n - 1)_{m_\ell r}}{r} \frac{(n - 1)_{m_1k} \cdots (n - 1)_{m_\ell k} - (n - 1 - r)_{m_1k} \cdots (n - 1 - r)_{m_\ell k}}{k}\frac{1}{n^{\tau r}}\frac{1}{n^{\tau k}} \\
& = \frac{(n - 1)_{m_1 r} \cdots (n - 1)_{m_\ell r}}{r}\cdot \#\left(\Gamma^k \cap \Gamma_i\right) \cdot \frac{1}{n^{\tau r}}\cdot \frac{1}{n^{\tau k}} \\
& = \mathcal{O}\left(\frac{1}{n}\right).
\end{align*}
Therefore, the total sum
$$\sum_{i \in \Gamma}\sum_{j \in \Gamma_i} p_ip_j = \mathcal{O}\left(\frac{K^2 }{n}\right).$$
For the last term in the upper bound in Lemma~\ref{lemma: chen-stein method} , we observe that $\mathbb{E}I_iI_j = 0$ if two cycles have shared vertices. 

Now, by Lemma~\ref{lemma: chen-stein method},
$$\mathbb{P}\left(Z_K = 0\right) \le e^{ - \lambda_{K}} + \mathcal{O}\left(\frac{K^2}{n}\right) \le \frac{1}{K + 1} + \mathcal{O}\left(\frac{K^2}{n}\right).$$
Sending $n \to \infty$ and noting that $K$ is arbitrary conclude the proof.
\end{proof}

\section{Proof of Theorem~\ref{theorem: main}} \label{section: proof of theorem}

Let $T$ be a tile with $\tau$ rows and $\sigma$ columns.
Define the \textbf{rank} of $T$ to be the largest $x$ such that there exist $x$ columns of $T$ with distinct $x\tau$ states.
We denote the rank of a tile as $\rank(T)$.
For example, the tiles
$$T_1 = 
\begin{matrix}
0 & 1 & 2 & 3\\
2 & 3 & 0 & 1\\
\end{matrix},\qquad
T_2 = 
\begin{matrix}
0 & 1 & 2 & 1\\
2 & 1 & 0 & 1\\
\end{matrix}.
$$
have $\rank(T_1) = 2$ and $\rank(T_2) = 1$.

As in~\cite{gl1}, we denote by $\mathcal{R}_{\tau, \sigma, n}^{(\ell)}$ as the set of tile of WRPS that has lag $\ell$.
% of a randomly selected rule.
Thus the set of simple WRPS is $\mathcal{R}_{\tau, \sigma, n}^{(0)}$.
%As in~\cite{gl1}, for a rule $f$ (which in this chapter is randomly selected), denote $\mathcal{R}_{\tau, \sigma, n}^{(0)} \subset \mathcal{R}_{\tau, \sigma, n}$ as the set of WRPS.
We also use the notation $\mathcal{R}_{\tau, \sigma, n}^{(0, y)} \subset \mathcal{R}_{\tau, \sigma, n}^{(0)}$
to denote the set of WRPS whose tile is simple and has rank $y$. 
We use $\mathcal{T}_{\tau, \sigma, n}$ to denote the set of all PS tiles; to be more precise, this is the set of 
all $\tau\times \sigma$ arrays $T$ with state space $\mathbb{Z}_n$
that satisfy properties 
1 and 2 in Lemma~\ref{lemma: properties of tile}, so that 
there exists a CA rule with a PS given by $T$.  
We also use $\mathcal{T}_{\tau, \sigma, n}^{(0)}$ and $\mathcal{T}_{\tau, \sigma, n}^{(0, y)}$ to denote the tiles in $\mathcal{T}_{\tau, \sigma, n}$ that are simple, and that are simple with rank $y$, respectively.

Our first step is to study the probability that $\mathcal{R}_{\tau, \sigma, n}^{(0, x)}$ is not empty, where $x = \sigma/\gcd(\tau, \sigma)$.
Before we advance, we state two lemmas on simple tiles.

\begin{lem} \label{lemma: rank of simple tiles}
Let $T$ be a simple tile. Then
\begin{enumerate}
\item $\rank(T) \geq \sigma/\gcd(\sigma, \tau)$;
\item $\rank(T) = y$ if and only if $s(T) = \tau y$.
In particular, $\rank(T) = \sigma/\gcd(\sigma, \tau)$ if and only if 
$s(T)=\tau \sigma/\gcd(\sigma, \tau) = \lcm(\sigma, \tau)$.
%$T$ contains $\tau \sigma/\gcd(\sigma, \tau) = \lcm(\sigma, \tau)$ distinct states.
\end{enumerate}
\end{lem}

\begin{proof}
By Lemma~\ref{lemma: simple tiles properties}, the states on each column of $T$ are distinct and two columns either share no common states or are circular shifts of each other.
As a result, $\rank(T)\ge s(T)/\tau$.
Together with Lemma~\ref{lemma: value of s(T)}, this proves (1)
and implication $(\Longrightarrow)$ of (2). The 
reverse implication in (2) follows from $s(T)\ge \tau \cdot \rank(T)$.  
\end{proof}

%To prove (1), assume $T$ to be a simple tile.
%By Lemma~\ref{lemma: value of s(T)}, $s(T)\ge \tau \sigma/\gcd(\tau, \sigma)$.
%By Lemma~\ref{lemma: simple tiles properties}, the states on each column of $T$ are distinct and two columns either share no common states or are circular shifts of each other.
%As a result, $T$ contains at least $\sigma/\gcd(\tau, \sigma)$ columns that have distinct states, giving the lower bound of the rank.
%
%To prove (2), 
%assume that $T$ has rank $\sigma/\gcd(\sigma, \tau)$. Then another application of Lemma~\ref{lemma: simple tiles properties} shows that $T$ contains $\tau \sigma/\gcd(\sigma, \tau) = \lcm(\sigma, \tau)$ distinct states.
%The other direction also follows from Lemma~\ref{lemma: simple tiles properties}.
%\end{proof}

In the sequel, we write $d = \gcd(\tau, \sigma)$, $k = \lcm(\sigma, \tau)$. 
By Lemma~\ref{lemma: rank of simple tiles}, $k$ is the number of distinct states in a simple tile with rank $x = \sigma/d$.
As before, $\varphi$ is the Euler totient function.
We index the tiles in $\mathcal{T}_{\tau, \sigma, n}^{(0, x)}$ in an arbitrary way.
Let
$$\mathfrak{T}_m = \left\{(T_i, T_j) \subset \mathcal{T}_{\tau, \sigma, n}^{(0, x)} \times \mathcal{T}_{\tau, \sigma, n}^{(0, x)}: i < j \text{ and }T_i, T_j \text{ have } m \text{ states in common}\right\}.$$
The following lemma gives the cardinality of these sets.

\begin{lem} \label{lemma: enumeration of simple tiles}
The following enumeration results hold:
\begin{enumerate}
\item the set $\mathcal{T}_{\tau, \sigma, n}^{(0, x)}$ has cardinality $\displaystyle \varphi(d)\binom{n}{k} (k - 1)!$;
\item if $m < k$, the set $\mathfrak{T}_m$ has cardinality
$$\frac{1}{2}\varphi(d)\binom{n}{k}(k - 1)! \varphi(d)\binom{k}{m} \binom{n - k}{k - m} (k - 1)! = \mathcal{O}\left(n^{2k - m}\right);$$
\item if $m = k$, the set $\mathfrak{T}_m$ has cardinality
$$\frac{1}{2}\varphi(d)\binom{n}{k}(k - 1)! \left(\varphi(d)(k - 1)! - 1\right) = \mathcal{O}\left(n^k\right).$$
\end{enumerate}
\end{lem}

\begin{proof}
Part (1) follows directly from Lemma~\ref{lemma: counting simple tile}. Then, part (2) follows from (1). Part (3) also follows from (1), after we note that once we select $T_i$, we have all $k$ colors fixed and we are not allowed to select $T_j$ 
equal to $T_i$.
\end{proof}

We will also need the following consequence of 
Theorem~\ref{theorem: simple label deciding probability}.
 
\begin{lem} \label{lemma: conditional probability of simple tiles}
Let $T$ be a simple tile and $\rank(T) = y$.
Let $A_0, \dots, A_{\sigma - 1}$ be the labels in $T$.
Then we have
$$\mathbb{P}\left(A_{i} \Rightarrow A_{i + 1}, \text{ for } i = 0, \dots, \sigma - 1 \bigm| A_{i} \to A_{i + 1}, \text{ for } i = 0, \dots, \sigma - 1\right) = \left(\frac{\tau}{n} + o \left(\frac{1}{n}\right)\right)^y.$$
\end{lem}
\begin{proof}
Assume that the $y$ columns with $y\tau$ states have indices
in $I\subset\{0,\ldots,\sigma-1\}$ and let those columns 
have labels $A_i$, $i\in I$. As $A_i$'s do not share any states, the events 
$\{A_i\rightarrow A_{i+1}\}$, $i\in I$
are independent, and so are 
$\{A_i\Rightarrow A_{i+1}\}$, $i\in I$.
We use Lemma~\ref{lemma: simple tiles properties} and 
Theorem~\ref{theorem: simple label deciding probability}
to get
\begin{align*}
&\mathbb{P}\left(A_{i} \Rightarrow A_{i + 1}, \text{ for } i = 0, \dots, \sigma - 1 \bigm| A_{i} \to A_{i + 1}, \text{ for } i = 0, \dots, \sigma - 1\right)\\
&=\frac{\mathbb{P}\left(A_{i} \Rightarrow A_{i + 1}, \text{ for } i \in I \right)}{\mathbb{P}\left(A_{i} \to A_{i + 1}, \text{ for } i \in I \right)}\\
&=\frac{\prod_{i \in I }\mathbb{P}\left(A_{i} \Rightarrow A_{i + 1}\right)}{\prod_{i \in I }\mathbb{P}\left(A_{i} \to A_{i + 1}\right)}\\
&= \left(\frac{n^\tau - (n - 1)^\tau}{n^\tau}\cdot \frac{1}{n^\tau}\right)^y\bigg/ \left(\frac{1}{n^\tau}\right)^y\\
&= \left(\frac{\tau}{n} + o \left(\frac{1}{n}\right)\right)^y,
\end{align*}
as desired.
\end{proof}

%Now, we are ready to study the probability that $\mathcal{R}_{\tau, \sigma, n}^{(0, x)}$ is non-empty.

Theorem~\ref{theorem: main} will now 
be established through next three propositions, the first 
one of which deals with existence of WRPS with zero lag and 
minimal rank $x = \sigma/d$.

\begin{prop}\label{prop: minimal rank}
We have
$$\mathbb{P}\left(\mathcal{R}_{\tau, \sigma, n}^{(0, x)} \neq \emptyset\right) = \frac{c(\tau, \sigma)}{n^x} + o\left(\frac{1}{n^x}\right),$$
for some constant $c(\tau, \sigma)$.
\end{prop}
\begin{proof}
We first find an upper bound by Markov inequality.

By Lemma~\ref{lemma: enumeration of simple tiles}, we have that $\# \mathcal{T}_{\tau, \sigma, n}^{(0, x)} = \displaystyle \varphi(d)\binom{n}{k} (k - 1)!$.
The probability that a tile in $\mathcal{T}_{\tau, \sigma, n}^{(0, x)}$ forms a PS is $1/n^{k}$ and the probability that the desired decidability, thus weak robustness, holds is $\left(\tau/n + o(1/n)\right)^x$ by Lemma~\ref{lemma: conditional probability of simple tiles}.
As a result, we have
$$\mathbb{E}\left(\#\mathcal{R}_{\tau, \sigma, n}^{(0, x)}\right) = \varphi(d)\binom{n}{k} (k - 1)!\frac{1}{n^{k}}\left( \frac{\tau}{n} + o\left(\frac{1}{n}\right)\right)^x = \frac{c(\tau, \sigma)}{n^x} + o\left(\frac{1}{n^{x}}\right),$$
as an upper bound. 

To find an asymptotically matching lower bound, we use the Bonferroni's inequality 
$$\mathbb{P}\left(\bigcup_i A_i\right) \ge \sum_i \mathbb{P}(A_i) - \sum_{i < j}\mathbb{P}\left(A_i\cap A_j\right).$$
Here, $A_i$ is the event that $T_i \in \mathcal{T}_{\tau, \sigma, n}^{(0, x)}$ is formed as a simple WRPS, for $\displaystyle i = 1, \dots, \varphi(d)\binom{n}{k}(k - 1)!$. 
Clearly, $\displaystyle \sum_i \mathbb{P}(A_i) = \mathbb{E}\left(\#\mathcal{R}_{\tau, \sigma, n}^{(0, x)}\right)$. 
Then it suffices to show that $\displaystyle \sum_{i < j}\mathbb{P}(A_i\cap A_j) = o\left(1/n^{x}\right)$.

For a pair of tiles $(T_i, T_j) \in \mathfrak{T}_m$, there are $2k - m$ different colors in $T_i \cup T_j$.
By Lemma~\ref{lemma: one more map}, there is at least one additional restriction on the number of maps.
Using this lemma, the enumeration result Lemma~\ref{lemma: enumeration of simple tiles}, and Lemma~\ref{lemma: conditional probability of simple tiles}, we have
\begin{align*}
\sum_{i<j}\mathbb{P}(A_i\cap A_j) 
& = \sum_{m = 0}^{k} \sum_{i < j} \mathbb{P}\left(A_i \cap A_j 
\cap\{ (T_i, T_j) \in \mathfrak{T}_m \}\right)\\
& = \sum_{m = 0}^{k} \mathcal{O}\left(n^{2k - m}\right) \frac{1}{n^{2k - m + 1}} \left(\frac{\tau}{n} + o\left(\frac{1}{n}\right)\right)^x \\
& = \mathcal{O}\left(\frac{1}{n^{x + 1}}\right).
\end{align*}
\end{proof}

Next, we consider \textit{all} simple tiles and show that among simple tiles, the WRPS with rank $x$ provide the dominant probability.

\begin{prop}\label{prop: simple wrps}
We have
$$\mathbb{P}\left(\mathcal{R}_{\tau, \sigma, n}^{(0)} \neq \emptyset\right) = \frac{c(\tau, \sigma)}{n^x} + o\left(\frac{1}{n^x}\right),$$ for the same constant $c(\tau, \sigma)$ as in Proposition~\ref{prop: minimal rank}.
\end{prop}

\begin{proof}
First, we note the following bounds for $\mathbb{P}(\mathcal{R}_{\tau, \sigma, n}^{(0)} \neq \emptyset)$,
$$\mathbb{P}\left(\mathcal{R}_{\tau, \sigma}^{(0, x)} \neq \emptyset\right) 
\le \mathbb{P}\left(\mathcal{R}_{\tau, \sigma, n}^{(0)} \neq \emptyset\right) 
\le \mathbb{P}\left(\mathcal{R}_{\tau, \sigma, n}^{(0, x)} \neq \emptyset\right) 
+ \sum_{y} \mathbb{P}\left(\mathcal{R}_{\tau, \sigma, n}^{(0, y)} \neq \emptyset\right),$$
where the last sum is over $y = \sigma/d^\prime$ for $d^\prime \mid \gcd(\tau, \sigma)$ and $d < \gcd(\tau, \sigma)$.
As $x < y$, we have from Lemmas~\ref{lemma: rank of simple tiles}--\ref{lemma: conditional probability of simple tiles},
\begin{align*}
\mathbb{P}\left(\mathcal{R}_{\tau, \sigma, n}^{(0, y)} \neq \emptyset\right) & \le \mathbb{E}\left(\#\mathcal{R}_{\tau, \sigma, n}^{(0, y)}\right) \\
& = \varphi(d_y)\binom{n}{k_y} (k_y - 1)!\frac{1}{n^{k_y}} \left(\frac{\tau}{n} + o\left(\frac{1}{n}\right)\right)^y\\
& = o\left(\frac{1}{n^x}\right),
\end{align*}
where, $k_y = \tau y$ is the number of states in a tile in $\mathcal{R}_{\tau, \sigma, n}^{(0, y)}$ and $d_y = \sigma/y$.
The conclusion now follows from Proposition~\ref{prop: minimal rank}.
\end{proof}

\begin{lem} \label{lemma: nonzero ell}
If $\ell>0$, then 
\begin{align*}
\mathbb{P}\left(\mathcal{R}_{\tau, \sigma, n}^{(\ell)} \neq \emptyset\right)
 = o\left(\frac{1}{n}\right).
\end{align*}
\end{lem}
\begin{proof}
For a fixed $\ell$, let $g_{\tau, \sigma}(s)$ count the number of tiles with periods $\tau$ and $\sigma$,   and $s$ different {\it fixed\/} states. By 
Theorem~\ref{theorem: general label deciding probability},
\begin{align*}
\mathbb{P}\left(\mathcal{R}_{\tau, \sigma, n}^{(\ell)} \neq \emptyset\right)
 \le \mathbb{E}\left(\#\mathcal{R}_{\tau, \sigma, n}^{(\ell)} \right) 
& = \sum_{s=1}^{\tau\sigma}\binom{n}{s} g_{\tau, \sigma, \ell}(s) \frac{1}{n^{s + \ell}} \cdot o(1) \\
& = o\left(\frac{1}{n^\ell}\right) 
 = o\left(\frac{1}{n}\right).
\end{align*}
 \end{proof}

Next, we extend Proposition~\ref{prop: simple wrps} to cover non-simple tiles. It is here that we 
impose the condition that $\sigma \mid \tau$.

\begin{prop} \label{prop: fix periods}
If $\sigma\mid \tau$, then 
$$\mathbb{P}\left(\mathcal{R}_{\tau, \sigma, n} \neq \emptyset\right) = \frac{c(\tau, \sigma)}{n} + o\left(\frac{1}{n}\right).$$
\end{prop}

\begin{proof}
First, note that $\sigma \mid \tau$ implies that $x = \sigma/\gcd(\tau, \sigma) = 1$ and as a result of Proposition~\ref{prop: simple wrps}, we have 
$$\mathbb{P}\left(\mathcal{R}_{\tau, \sigma, n}^{(0)} \neq \emptyset\right) = \frac{c(\tau, \sigma)}{n} + o\left(\frac{1}{n}\right).$$
%Now consider a general $\ell > 0$.  Let $g_{\tau, \sigma, \ell}(s)$ count the number of tiles with period $\tau$ and $\sigma$, $m - s = \ell$, and $s$ different states. By 
%Theorem~\ref{theorem: general label deciding probability},
%\begin{align*}
%\mathbb{P}\left(\mathcal{R}_{\tau, \sigma, n}^{(\ell)} \neq \emptyset\right)
% \le \mathbb{E}\left(\#\mathcal{R}_{\tau, \sigma, n}^{(\ell)} \right) 
%& = \sum_{s=1}^{\tau\sigma}\binom{n}{s} g_{\tau, \sigma, \ell}(s) \frac{1}{n^{s + \ell}} \cdot o(1) \\
%& = o\left(\frac{1}{n^\ell}\right) 
% = o\left(\frac{1}{n}\right).
% 
The desired result now follows from the bounds
$$ 
\mathbb{P}\left(\mathcal{R}_{\tau, \sigma, n}^{(0)} \neq \emptyset\right) 
\le \mathbb{P}\left(\mathcal{R}_{\tau, \sigma, n} \neq \emptyset\right) 
\le \sum_{\ell = 0}^{\tau \sigma} \mathbb{P} \left(\mathcal{R}_{\tau, \sigma, n}^{(\ell)} \neq \emptyset\right)$$
and Lemma~\ref{lemma: nonzero ell}.
\end{proof}

%Now, we are ready to prove the main result, Theorem~\ref{theorem: main}.

\begin{proof}[Proof of Theorem~\ref{theorem: main}]
If $\sigma\nmid \tau$, then $x = \sigma/\gcd(\tau, \sigma) >1$, and by Proposition~\ref{prop: simple wrps} and Lemma~\ref{lemma: nonzero ell}, 
$$
\mathbb{P}\left(\mathcal{R}_{\tau, \sigma, n} \neq \emptyset\right) 
\le \mathbb{P}\left(\mathcal{R}_{\tau, \sigma,n}^{( 0)} \neq \emptyset\right) 
+ \sum_{\ell = 1}^{\tau \sigma} \mathbb{P} \left(\mathcal{R}_{\tau, \sigma,n}^{(\ell)} \neq \emptyset\right) 
= \frac{c(\tau, \sigma)}{n^x} + o\left(\frac{1}{n}\right) 
= o\left(\frac{1}{n}\right).$$
These bounds, together with Proposition~\ref{prop: fix periods}, now give the desired result: 
\begin{align*}
\frac{c(\Tau, \Sigma)}{n} + o\left(\frac{1}{n}\right)
& = \sum_{\sigma \mid \tau}\mathbb{P}(\mathcal{R}_{\tau, \sigma, n} \neq \emptyset)\\
&\le \mathbb{P}(\mathcal{R}_{\Tau, \Sigma, n} \neq \emptyset)\\
&\le \sum_{\sigma \mid \tau}\mathbb{P}(\mathcal{R}_{\tau, \sigma, n} \neq \emptyset) 
+ \sum_{\sigma\nmid \tau}\mathbb{P}(\mathcal{R}_{\tau, \sigma, n} \neq \emptyset)
\le \frac{c(\Tau, \Sigma)}{n} + o\left(\frac{1}{n}\right).
\end{align*}
\end{proof}

\section{Discussion}
Inspired by~\cite{gravner2012robust}, we prove that the probability that a randomly chosen CA has a weakly robust periodic solution with periods in the finite set $\Tau\times\Sigma$ is asymptotically $c(\Tau, \Sigma)/n$, provided that $\Tau \times \Sigma$ contains a pair $(\tau, \sigma)$ with $\sigma \mid \tau$. A natural first question is whether the divisibility 
condition may be removed.  

\begin{ques}\label{question: generalization}
Let $\mathcal{R}_{\tau, \sigma, n}$ be the set of WRPS with periods $\tau$ and $\sigma$ from a random rule $f$.
Do we have 
$$\mathbb{P}(\mathcal{R}_{\tau, \sigma, n} \neq \emptyset) = \frac{c(\tau, \sigma)}{n^x}  + o\left(\frac{1}{n^x}\right),$$ 
where $x = \sigma/\gcd(\tau, \sigma)$?
\end{ques}

A possible strategy to answer Question~\ref{question: generalization} affirmatively is through proving the 
following two conjectures, the first of 
which provides a lower bound of the rank of a tile.
Recall that $x=\sigma/\gcd(\tau, \sigma)$. 

\begin{con}\label{conjecture: lower bound of simple labels}
Let $T$ be a tile of a WRPS of period $\tau$ and $\sigma$ and $\ell = p(T) - s(T)$. Then $\rank(T) \ge x - \ell$.
\end{con}

We recall that a tile of a WRPS satsifies the
properties stated in Lemmas~\ref{lemma: properties of tile} 
and~\ref{lemma: irreducibility of labels}. The next 
conjecture presents an asymptotic property similar to the one in Theorem~\ref{theorem: general label deciding probability}. In its 
formulation, we assume 
validity of Conjecture~\ref{conjecture: lower bound of simple labels}:
for a tile $T$ of a WRPS, we let  
$I=I(T)\subset \{0,\ldots,\sigma-1\}$ be the index set with $\#I=x-\ell$, such that 
the labels indexed by $I$ are the leftmost $x-\ell$ labels 
without a repeated state. 
%We claim that an affirmative answer to the above question can be given if the above Conjecture~\ref{conjecture: lower bound of simple labels} as well as the following one both hold.

\begin{con}\label{conjecture: extra}
Assume that $T$ is a tile of a WRPS. Then there exists a label $A_j$ with index $j\notin I$ so 
that  
$$\mathbb{P}\left(A_j \Rightarrow A_{j + 1} \bigm|
\{A_i\Rightarrow A_{i+1} \text{ for all }i\in I\}\right) = o(1).$$ 
\end{con}

%Let $T$ be a tile of rank $r$. and $A_{i_0}, \dots, A_{i_{x - 1}}$ be the simple labels without repeated state as in Conjecture~\ref{conjecture: lower bound of simple labels}.
%Then there is at least one more label $A_j$ such that
%$$\mathbb{P}\left(A_j \Rightarrow A_{j + 1} \bigm| \text{decidabilities of } A_{i_0}, \dots, A_{i_{x - 1}} \text{ are satisfied }\right) = o(1).$$ 

%giving an $o(1)$ of probability. 
%This label could be one of the following cases:
%Case 1. A label that is independent from the above independent and simple labels but itself is non-simple;
%Case 2. A label that is not independent but simple;
%Case 3. A label that is not independent and non-simple.

If there exists a label $j$ that does not share any state with $A_{i}$, for any $i\in I$, the conjecture can be proved in the same way as Theorem~\ref{theorem: general label deciding probability}.
To see how Question~\ref{question: generalization} is settled  in the case that both of the conjectures are satisfied, use again the bounds
$$
\mathbb{P}\left(\mathcal{R}_{\tau, \sigma, n}^{(0)} \neq \emptyset\right) 
\le \mathbb{P}\left(\mathcal{R}_{\tau, \sigma, n} \neq \emptyset\right)
\le \mathbb{P}\left(\mathcal{R}_{\tau, \sigma, n}^{(0)} \neq \emptyset\right) 
+ \sum_{\ell}\mathbb{E}\left(\#\mathcal{R}_{\tau, \sigma, n}^{(\ell)}\right),\\
$$
and then, with $g_{\tau, \sigma}(s)$ as in the proof of 
Lemma~\ref{lemma: nonzero ell}, and using 
Lemma~\ref{lemma: conditional probability of simple tiles},
$$\mathbb{E}\left(\#\mathcal{R}_{\tau, \sigma, n}^{(\ell)}\right)
= \sum_{s = 1}^{\tau\sigma} \binom{n}{s} g_{\tau, \sigma}(s)
\frac{1}{n^m} \cdot \mathcal O\left(\frac{1}{ n^{x - \ell}}\right)\cdot o(1) = o\left(\frac{1}{n^x}\right).$$

To provide some modest evidence for the validity of   Conjecture~\ref{conjecture: lower bound of simple labels}, we prove that it holds when $\sigma = 2$ or $\tau = 2$.
Conjecture~\ref{conjecture: extra} remains open even in 
these cases.
We begin by the following lemma.

\begin{lem}\label{lemma: lag is nondecreasing}
Let $T$ be a tile of a WRPS with $\sigma=2$ and odd $\tau$. Fix an arbitrary row as the 0th row.
Let $\mathcal{M}_t = \{\text{maps up to } t \text{ th row}\}$, $\mathcal{S}_t = \{\text{states up to } t \text{ th row}\}$ and $\ell_t = \#\mathcal{M}_t - \#\mathcal{S}_t$, for $t = 0, 1, \dots, \tau - 1$. Assume the $(t + 1)$th row of the tile is 
$ab$. Then: 
\begin{enumerate}
\item if $a \in \mathcal{S}_t$ and $b \in \mathcal{S}_t$, 
$\ell_{t+1}-\ell_t=2$; 
\item if exactly one of $a$ and $b$ is in $\mathcal{S}_t$, then $\ell_{t+1}-\ell_t= 1$; and
\item if $a \notin \mathcal{S}_t$ and $b \notin \mathcal{S}_t$, $\ell_{t+1}-\ell_t=0$.
\end{enumerate}
%Then $\ell_t$ is non-decreasing in $t$. 
\end{lem}

\begin{proof}
Write $\ell_{t + 1} - \ell_t = (\#\mathcal{M}_{t + 1} - \#\mathcal{M}_t) - (\#\mathcal{S}_{t + 1} - \#\mathcal{S}_t)$.  
Observe that $a \neq b$, as otherwise the spatial period of the tile is reducible.
In addition,  $(a, b) \notin \mathcal{M}_t$, as otherwise $T$ is temporally reducible,  and $(b, a) \notin \mathcal{M}_t$, as 
otherwise $\tau$ is even. 
Hence, $\#\mathcal{M}_{t + 1} - \#\mathcal{M}_t = 2$, 
which implies the claim.
\end{proof}

\begin{proof}[Proof of Conjecture~\ref{conjecture: lower bound of simple labels} when $\sigma = 2$]
If $\tau$ is even, we need to show that 
$\rank(T)\ge 1 - \ell$. 
This is trivial if $\ell \ge 1$,  and follows from Lemma~\ref{lemma: simple tiles properties} when $\ell = 0$.

If $\tau$ is odd, we must show that  $\rank(T)\ge 2 - \ell$. 
We may assume $\ell=1$ as otherwise this is immediate (as above).
%If $\ell = 0$, $T$ is simple and the result again follows from Lemma~\ref{lemma: simple tiles properties}. 
Then there exists exactly one $t \in \{0, \dots, \tau - 1\}$
at which Case 2 of Lemma~\ref{lemma: lag is nondecreasing} happens, and otherwise Case 3 happens. If $a \in \mathcal{S}_{t}$, then column with $b$ has 
no repeated state, and vice versa.
\end{proof}

\begin{proof}[Proof of Conjecture~\ref{conjecture: lower bound of simple labels} when $\tau = 2$]
We will prove this for any tile that satisfies the
properties stated in Lemmas~\ref{lemma: properties of tile} 
and~\ref{lemma: irreducibility of labels}. 
We assume that no two different labels of $T$ are rotations of each other; otherwise the argument is similar.

We use induction on the lag.
If $\ell(T)= 0$, $T$ is simple and Lemma~\ref{lemma: simple tiles properties} applies.
Suppose now the statement is true for any tile $T$ with  
$\ell(T) = \ell\ge 0$.
Now, consider a tile $T$ with   $\ell(T) = \ell + 1$.
As $\ell(T) \ge 1$, there is at least one repeated state, say $a$.
Consider two appearance of $a$ and its neighbors:
$$bac \quad \text{and} \quad  b^\prime a c^\prime.$$
As $\tau=2$ and $T$ has no rotated columns, $b \neq b^\prime$ and $c \neq c^\prime$.
Now replace the $a$ in $bac$ by an arbitrary state not represented in $T$, say $z$, and denote the new tile by $T^\prime$.
Note that $T^\prime$ also satisfies the properties in Lemmas~\ref{lemma: properties of tile} and~\ref{lemma: simple tiles properties}.
Moreover, $p(T^\prime) = p(T)$ and $s(T^\prime) = s(T) + 1$ imply that $\ell (T^\prime) = \ell$.
By inductive hypothesis, $\rank(T^\prime) \ge \sigma/\gcd(\sigma, \tau) - \ell$. 
Among $\rank(T^\prime)$ labels of $T'$ without 
a repeated state, at most one has the state $z$.
Excluding this label, if necessary, we conclude that 
$\rank(T)\ge \sigma/\gcd(\sigma, \tau) - (\ell + 1)$.
\end{proof}

Besides the above two special cases, we are also able to prove Conjecture~\ref{conjecture: lower bound of simple labels} for a special class of tiles, which may give a hint about the general case.
Within $T$, fix an arbitrary row as the $0$th row and find the smallest $\tilde{\tau}$ such that \texttt{row}$_{\tilde{\tau}}$ is a cyclic permutation of \texttt{row}$_0$.
It is likely that such $\tilde{\tau}$ does not exist, in which case define $\tilde{\tau} = \tau$. 
We call $T$ \textbf{semi-simple} if $p(T) = \tilde{\tau}\sigma$; i.e., within the first $\tilde{\tau}$ rows in $T$, there are no repeated states. We omit the proof of our last lemma, as 
it is very similar to the argument above. 

\begin{lem} \label{lemma: rank in semi-simple tile}
A semi-simple tile $T$ has rank at least $\sigma/\gcd(\tau, \sigma) - \ell$.
\end{lem}
\iffalse
\begin{proof}
The proof is similar to the above one and we thus omit the details here.

Use induction on $\ell$.
For $\ell = 0$, $T$ is simple and the result follows from Lemma~\ref{lemma: simple tiles properties}.
Suppose the statement is true for any semi-simple tile $T$ satisfying $\ell(T) = \ell$.
Now, consider a semi-simple tile $T$, in which $p(T) - s(T) = \ell(T) = \ell + 1$.
Using the condition that $T$ is semi-simple, we have $s(T) = \tilde{\tau}\sigma - (\ell + 1) < \tilde{\tau} \sigma$.
This implies that within the first $\tilde{\tau}$ rows, there exists a state appearing at least twice and we may denote this state by $a$.
For these two appearance, consider the neighbors of $a$:
$$bac \quad \text{and} \quad  b^\prime a c^\prime.$$
As $T$ is semi-simple, $b \neq b^\prime$ and $c \neq c^\prime$.
Now, in $T$, replace every $a$ that has neighbor $b$ and $c$ with a state outside $T$, say $x$ and denote the new tile by $T^\prime$.
Note that $T^\prime$ is also a WRPS tile, i.e., it satisfies Lemma~\ref{lemma: properties of tile} and that ``the period of any label cannot be reduced''.
In addition, $p(T^\prime) = p(T)$ and $s(T^\prime) = s(T) + 1$ together imply that $\ell (T^\prime) = \ell$.
By inductive hypothesis, $T^\prime$ has $\sigma/\gcd(\sigma, \tau) - \ell$ independent and simple labels. 
Among them, at most one label has the state $z$, otherwise contradicting to independence.
Now, excluding this label, the other $\sigma/\gcd(\sigma, \tau) - (\ell + 1)$ are independent and simple labels in the original tile $T$.

\end{proof}
\fi

\section*{Acknowledgements}
Both authors were partially supported by the NSF grant DMS-1513340.
JG was also supported in part by the Slovenian Research Agency (research program P1-0285). 

\bibliography{references}

\begin{thebibliography}{10}

\bibitem{barbour1992poisson}
Andrew~D Barbour, Lars Holst, and Svante Janson.
\newblock {\em Poisson approximation}.
\newblock The Clarendon Press, 1992.

\bibitem{gravner2011one}
Janko Gravner and David Griffeath.
\newblock The one-dimensional exactly 1 cellular automaton: replication,
  periodicity, and chaos from finite seeds.
\newblock {\em Journal of Statistical Physics}, 142(1):168--200, 2011.

\bibitem{gravner2012robust}
Janko Gravner and David Griffeath.
\newblock Robust periodic solutions and evolution from seeds in one-dimensional
  edge cellular automata.
\newblock {\em Theoretical Computer Science}, 466:64, 2012.

\bibitem{gl3}
Janko Gravner and Xiaochen Liu.
\newblock Maximal temporal period of a periodic solution generated by a
  one-dimensional cellular automaton.
\newblock {\em arXiv preprint arXiv:1909.06915}, 2019.

\bibitem{gl2}
Janko Gravner and Xiaochen Liu.
\newblock One-dimensional cellular automata with random rules: longest temporal
  period of a periodic solution.
\newblock {\em arXiv preprint arXiv:1909.06914}, 2019.

\bibitem{gl1}
Janko Gravner and Xiaochen Liu.
\newblock Periodic solutions of one-dimensional cellular automata with random
  rules.
\newblock {\em arXiv preprint arXiv:1909.06913}, 2019.

\bibitem{xiaochen}
Xiaochen Liu.
\newblock {\em Cellular automata with random rules}.
\newblock PhD thesis, University of California, Davis, in preparation, 2020.

\bibitem{ross2011fundamentals}
Nathan Ross.
\newblock Fundamentals of {Stein}'s method.
\newblock {\em Probability Surveys}, 8:210--293, 2011.

\bibitem{moh1990number}
Moh'd Z.~Abu Sbeih.
\newblock On the number of spanning trees of {$K_n$ and $K_{m, n}$}.
\newblock {\em Discrete mathematics}, 84(2):205--207, 1990.

\bibitem{strogatz2001nonlinear}
Stephen Strogatz.
\newblock {\em Nonlinear dynamics and chaos: with applications to physics,
  biology, chemistry, and engineering}.
\newblock Westview Press, 2015.

\end{thebibliography}

\end{document}